\RequirePackage[l2tabu, orthodox]{nag}
\documentclass[final,11pt]{article}

\usepackage{tikz,bm}
\usetikzlibrary{arrows.meta}

\newcommand{\myDelta}{\mathrm{\Delta}}
\newcommand{\myGamma}{\mathrm{\Gamma}}

\usepackage{enumitem}
\newlist{case}{enumerate}{3} 
\setlist[case]{align=left, 
                 listparindent=\parindent, 
                 parsep=\parskip, 
                 font=\normalfont\bfseries, 
                 leftmargin=0pt, 
                 labelwidth=0pt, 
                 itemindent=.4em,labelsep=.4em, 
                 partopsep=0pt, 
                 }
\setlist[case,1]{label=Case~\arabic*:,ref=\arabic*}
\setlist[case,2]{label=Case~\thecasei.\arabic*:,ref=\thecasei.\arabic*}

\usepackage[final]{microtype}
\makeatletter
\def\MT@register@subst@font{\MT@exp@one@n\MT@in@clist\font@name\MT@font@list
   \ifMT@inlist@\else\xdef\MT@font@list{\MT@font@list\font@name,}\fi}
\makeatother 

\usepackage[small]{titlesec}

\usepackage{amsmath, amssymb, amsfonts, amsthm, mathtools}

\usepackage{fixmath}
\mathtoolsset{centercolon}

\usepackage[latin1]{inputenc}
\usepackage[T1]{fontenc}

\usepackage[draft=false,pdftex]{hyperref}
\hypersetup{
    bookmarks=true,         
    unicode=false,          
    pdftoolbar=true,        
    pdfmenubar=true,        
    pdffitwindow=false,     
    pdfstartview={FitH},    
    pdftitle={Shortest directed networks in the plane},    
    pdfauthor={Alastair Maxwell and Konrad J. Swanepoel},     
    pdfnewwindow=false,      
    colorlinks=true,       
    linkcolor=black,          
    citecolor=black,        
    filecolor=black,      
    urlcolor=black           
}

\theoremstyle{plain}
\newtheorem{theorem}{Theorem}
\newtheorem{lemma}[theorem]{Lemma}
\newtheorem{corollary}[theorem]{Corollary}
\newtheorem{conjecture}[theorem]{Conjecture}
\newtheorem{proposition}[theorem]{Proposition}

\theoremstyle{definition}

\newtheorem{problem}[theorem]{Problem}

\newcommand{\setbuilder}[2]{\{#1:#2\}}
\newcommand{\set}[1]{\left\{#1\right\}}
\newcommand{\opair}[2]{(#1,#2)}
\newcommand{\arc}[2]{#1\mathord{\rightarrow}#2}
\newcommand{\norm}[1]{\lVert#1\rVert}
\newcommand{\ipr}[2]{\left\langle #1, #2 \right\rangle}
\newcommand{\abs}[1]{\left\lvert#1\right\rvert}
\newcommand{\myangle}{\sphericalangle}
\newcommand{\numbersystem}[1]{\mathbb{#1}}
\newcommand{\bE}{\numbersystem{E}}
\newcommand{\bR}{\numbersystem{R}}
\DeclareMathOperator{\conv}{conv}

\title{Shortest directed networks in the plane}
\author{Alastair Maxwell\thanks{C/O Konrad Swanepoel, Department of Mathematics, London School of Economics and Political Science, Houghton Street, London WC2A 2AE, United Kingdom. This work is partially based on a dissertation of the first author written as part of an MSc in Applicable Mathematics at the LSE.}
\and 
Konrad J.\ Swanepoel\thanks{Department of Mathematics, London School of Economics and Political Science, Houghton Street, London WC2A 2AE, United Kingdom. Email: 
    \href{mailto:k.swanepoel@lse.ac.uk}{k.swanepoel@lse.ac.uk}}
}
\date{}

\begin{document}
\maketitle

\begin{abstract}
Given a set of sources and a set of sinks as points in the Euclidean plane, a \emph{directed network} is a directed graph drawn in the plane with a directed path from each source to each sink.
Such a network may contain nodes other than the given sources and sinks, called Steiner points.
We characterize the local structure of the Steiner points in all shortest-length directed networks in the Euclidean plane.
This characterization implies that these networks are constructible by straightedge and compass.
Our results build on unpublished work of Alfaro, Campbell, Sher, and Soto from 1989 and 1990.
Part of the proof is based on a new method that uses other norms in the plane.
This approach gives more conceptual proofs of some of their results, and as a consequence, we also obtain results on shortest directed networks for these norms.
\end{abstract}



\section{Introduction}
In the well-studied Euclidean Steiner problem, a finite set of points in the plane is given, and the problem is to find a shortest network interconnecting all points.
Its fascinating history is given a definitive treatment by Brazil, Graham, Thomas and Zachariasen \cite{BGTZ}.
What distinguishes this problem from the well-known Minimal Spanning Tree problem, is that such a network (necessarily a tree) may contain new points, called \emph{Steiner points}.
The degree of such a Steiner point is always $3$, and the angle between any two incident edges is $120^\circ$.
In fact, any such tree is constructible using straightedge and compass.
Despite the simplicity of this local structure, it is NP-hard to compute such a shortest network.
On the other hand, there is an exact algorithm (GeoSteiner) that can feasibly compute these networks for given point sets of size in the thousands.
For more detail, see Chapter~1 of Brazil and Zachariasen~\cite{BZ}.

A directed version of this problem was introduced by Frank Morgan for undergraduate research at Williams College in the late 1980s \cite{Alfaro-Thesis, Alfaro, ACSS}.
In this problem, a set of sources and a set of sinks in the plane are given, and the object is to find a shortest directed network containing a directed path from each source to each sink.
As these directed networks are not necessarily trees, they are much harder to study, and even their existence is non-trivial~\cite{Alfaro}.
See Fig.~\ref{fig33}(b) for an example of a shortest directed network that has a cycle.
There is no known algorithm for finding such networks, and as a first step, the local structure of such networks has to be described.
There are some partial results in~\cite{Alfaro-Thesis, ACSS}.
The main contribution of this paper is to complete their results with a complete description of the local structure of the directed edges incident to a Steiner point in a shortest directed network (Theorem~\ref{thm:main}).
We also find new proofs of some of their results.
Our characterization easily implies that any shortest directed network in the plane is constructible by straightedge and compass (Corollary~\ref{cor:main}), and that there exists an algorithm (even though prohibitively inefficient) that can compute them (Corollary~\ref{cor:algorithm}).

We make use of norms other than the Euclidean, which on the one hand gives conceptually simpler proofs of some of the results in \cite{Alfaro-Thesis}, and on the other hand also give examples of shortest directed networks for these norms.

\section{Main results}
The digraphs $G=(V,E)$ in this paper will be \emph{simple}, that is, they are without loops.
We call the elements of $V$ \emph{nodes}. The elements of $E$, which we call \emph{directed edges} or just \emph{edges}, 
are directed pairs of distinct nodes, and denoted by $\arc{x}{y}$, where $x$ is the \emph{tail} and $y$ the \emph{head} of the directed edge.
The \emph{indegree} of a node $x$ is the number $\deg^-(x)$ of directed edges in $E$ with head $x$, and the \emph{outdegree} the number $\deg^+(x)$ of directed edges in $E$ with tail $x$.
The \emph{degree} of $x$ is the ordered pair $\deg(x)=(\deg^-(x),\deg^+(x))$.

Given a digraph $G=(V,E)$ and nodes $a,b\in V$, a \emph{directed path} from $a$ to $b$ is a finite sequence of distinct vertices $a=x_1,x_2,\dots,x_n=b$ ($n\geq 1$) such that $\arc{x_i}{x_{i+1}}\in E$ for each $i=1,\dots,n-1$.
We allow a path with a single vertex.
We call any directed path in $G$ from $a$ to $b$ an \emph{$(a,b)$-path}.
Given subsets $A,B\subseteq V$, we say that $G$ is an \emph{$(A,B)$-network} if $G$ contains an $(a,b)$-path for each $a\in A$ and $b\in B$.
We do not require $A$ and $B$ to be disjoint.
The nodes in $V\setminus (A\cup B)$ are called the \emph{Steiner points} of the $(A,B)$-network.
We call an $(A,B)$-network \emph{simple} if for each Steiner point $s\in V$, $\deg^-(s)\geq 1$, $\deg^+(s)\geq 1$ and $\deg^-(s)+\deg^+(s)\geq 3$.

Let $X=(\bR^d,\norm{\cdot})$ be a $d$-dimensional normed space.
A \emph{directed network} or \emph{geometric digraph} in $X$ is an embedding of a digraph $G=(V,E)$ into $X$ such that each node in $V$ is represented as a point in $X$ and where each directed edge is drawn as a straight-line segment from its tail to its head.
We allow distinct nodes in $V$ to be represented by the same point in $X$.
The \emph{length} of a directed edge $\arc{x}{y}$ in a directed network is its length in the norm $\norm{\arc{x}{y}}:=\norm{x-y}$.
The \emph{length} of a network $G$ is the sum of the lengths of its directed edges and denoted $\norm{G} := \sum_{\arc{x}{y}\in E} \norm{\arc{x}{y}}$.
We will sometimes work with the same network in $\bR^d$ but with more than one norm.
To avoid confusion, we will always use subscripts to distinguish between different norms.
These norms are all introduced in Section~\ref{sec:basic}.

Note that any $(A,B)$-network in $X$ can be modified into a simple $(A,B)$-network without increasing its length.
Indeed, given any Steiner point $s$ of the $(A,B)$-network $G$, if $\deg^-(s)=0$ or $\deg^+(s)=0$, then $s$ and its incident directed edges can be removed from $G$, and the new directed graph remains an $(A,B)$-network.
Also, if $\deg^-(s)=\deg^+(s)=1$, then $s$ and its incident directed edges $\arc{x}{s}$ and $\arc{s}{y}$ can be replaced by a single directed edge $\arc{x}{y}$ to obtain an $(A,B)$-network $G'$ with $\norm{G'}\leq \norm{G}$.
By applying these two procedures repeatedly, we obtain a simple $(A,B)$-network after finitely many steps.

Given finite sets $A$ and $B$ of points from $X$, a \emph{shortest $(A,B)$-network} is an $(A,B)$-network of minimum length among all $(A,B)$-networks in $X$.
From the remarks in the previous paragraph, we only have to consider simple $(A,B)$-networks when finding shortest ones.
It is not obvious that any finite sets of points $A$ and $B$ in a finite-dimensional normed space $X$ has a shortest $(A,B)$-network, as an $(A,B)$-network can have cycles and there is no immediate upper bound for the number of Steiner points in such a network.
However, it has been shown that the number of Steiner points in a simple $(A,B)$-network is bounded by $O(|A|+|B|)$ in the Euclidean plane \cite{Alfaro} and $O(|A|^2|B|+|A||B|^2)$ in any normed space (or indeed, any metric space)~\cite{SwaJCMCC}.
This, together with a compactness argument, shows that for any given finite subsets $A$ and $B$ of a finite-dimensional normed space, there always exists at least one shortest $(A,B)$-network.

Partial results on the local structure of Steiner points in a shortest $(A,B)$-network in the Euclidean plane can be found in \cite{Alfaro-Thesis, ACSS}.
Our main result is a completion of these partial results into a full characterization.
\begin{theorem}\label{thm:main}
The following is a complete list of all possibilities for the local geometric structure of a Steiner point $s$ in a shortest $(A,B)$-network $G$ in the Euclidean plane.
\begin{enumerate}
\item\label{12} $\deg(s)=(1,2)$ or $(2,1)$.
The three directed edges incident to $s$ are pairwise at $120^\circ$ angles (Fig.~\ref{fig12}).
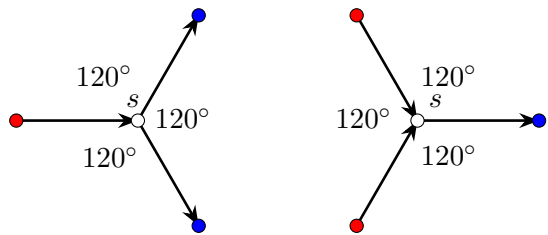
\begin{figure}[h]
\definecolor{ffffff}{rgb}{1.,1.,1.}
\definecolor{ffqqqq}{rgb}{1.,0.,0.}
\definecolor{qqqqff}{rgb}{0.,0.,1.}
\centering
\begin{tikzpicture}[line cap=round,line join=round,>=Stealth,x=0.35cm,y=0.35cm]
\draw[color=black] (1.4,0.1) node {$120^\circ$};
\draw[color=black] (-1.6,1.7) node {$120^\circ$};
\draw[color=black] (-1.3668525851623285,-1.3881666854846548) node {$120^\circ$};
\draw[color=black] (11.5,1.7) node {$120^\circ$};
\draw[color=black] (8.255483717475798,0.1) node {$120^\circ$};
\draw[color=black] (11.5,-1.3191893643187904) node {$120^\circ$};
\draw [->,line width=1pt] (-0.30940107675850337,0.) -- (2.,4.);
\draw [->,line width=1pt] (-0.30940107675850337,0.) -- (2.,-4.);
\draw [->,line width=1pt] (-4.9282032302755105,0.) -- (-0.3094010767585029,0.);
\draw [->,line width=1pt] (8.,4.) -- (10.309401076758503,0.);
\draw [->,line width=1pt] (8.,-4.) -- (10.309401076758503,0.);
\draw [->,line width=1pt] (10.309401076758503,0.) -- (14.92820323027551,0.);
\draw [fill=ffffff] (-0.30940107675850337,0.) circle (2.5pt);
\draw[color=black] (-0.5,0.75) node {$s$};
\draw [fill=ffffff] (10.309401076758503,0.) circle (2.5pt);
\draw[color=black] (11,0.75) node {$s$};
\draw [fill=qqqqff] (2.,-4.) circle (2.5pt);
\draw [fill=qqqqff] (2.,4.) circle (2.5pt);
\draw [fill=ffqqqq] (-4.9282032302755105,0.) circle (2.5pt);
\draw [fill=ffqqqq] (8.,4.) circle (2.5pt);
\draw [fill=ffqqqq] (8.,-4.) circle (2.5pt);
\draw [fill=qqqqff] (14.92820323027551,0.) circle (2.5pt);
\end{tikzpicture}
\caption{$\deg(s)=(1,2)$ or $(2,1)$}\label{fig12}
\end{figure}

\item\label{22} $\deg(s)=(2,2)$. One of the following two cases:
\begin{enumerate}[label={\textup{(\alph*)}},ref={(\alph*)}]
\item\label{a} Opposite pairs of directed edges lie on two straight lines, with directed edges alternating between incoming and outgoing (Fig.~\ref{fig22}(a)).

\item\label{b} Opposite pairs of directed edges lie on two straight lines, directed edges do not alternate between incoming and outgoing, and the angles between the two incoming directed edges and between the two outgoing directed edges are $\geq 120^\circ$ (Fig.~\ref{fig22}(b)).
\end{enumerate}

\begin{figure}[h]
\definecolor{ffffff}{rgb}{1.,1.,1.}
\definecolor{qqqqff}{rgb}{0.,0.,1.}
\definecolor{ffqqqq}{rgb}{1.,0.,0.}
\centering
\begin{tikzpicture}[line cap=round,line join=round,>=Stealth,x=0.35cm,y=0.35cm]
\draw (-5,-8) node {(a)};
\draw [fill=qqqqff] (4.,-8.) circle (2.5pt);
\draw (10.5,-8) node {(b)};
\draw[color=black] (14.4,-13.3) node {$\geq 120^\circ$};
\draw[color=black] (19.4,-13.3) node {$\geq 120^\circ$};
\draw [->,line width=1pt] (-2.,-8.) -- (0.9784613544229659,-13.295753308926642);
\draw [->,line width=1pt]  (4.04786712635805,-18.75320728899109) -- (0.9784613544229659,-13.295753308926642);
\draw [->,line width=1pt]  (0.9784613544229659,-13.295753308926642) -- (-2.056625796820976,-18.615252646659357);
\draw [->,line width=1pt] (0.9784613544229659,-13.295753308926642) -- (4.,-8.);
\draw [->,line width=1pt] (13.893695095165082,-8.009989517407663) -- (16.922390036516978,-13.359261991288326);
\draw [->,line width=1pt]  (13.945428086039483,-18.718718628408155) -- (16.922390036516978,-13.359261991288326);
\draw [->,line width=1pt] (16.922390036516978,-13.359261991288326) -- (19.893695095165082,-8.009989517407663);
\draw [->,line width=1pt]  (16.922390036516978,-13.359261991288326) -- (20.015432348635578,-18.822184610156953);
\draw [fill=ffffff] (0.9784613544229659,-13.295753308926642) circle (2.5pt);
\draw[color=black] (17.0,-12) node {$s$};
\draw [fill=ffffff] (16.922390036516978,-13.359261991288326) circle (2.5pt);
\draw[color=black] (1.8,-13.131555613973136) node {$s$};
\draw [fill=ffqqqq] (-2.,-8.) circle (2.5pt);
\draw [fill=qqqqff] (-2.056625796820976,-18.615252646659357) circle (2.5pt);
\draw [fill=ffqqqq] (4.04786712635805,-18.75320728899109) circle (2.5pt);
\draw [fill=ffqqqq] (13.893695095165082,-8.009989517407663) circle (2.5pt);
\draw [fill=ffqqqq] (13.945428086039483,-18.718718628408155) circle (2.5pt);
\draw [fill=qqqqff] (20.015432348635578,-18.822184610156953) circle (2.5pt);
\draw [fill=qqqqff] (19.893695095165082,-8.009989517407663) circle (2.5pt);
\end{tikzpicture}
\caption{$\deg(s)=(2,2)$: two possibilities}\label{fig22}
\end{figure}
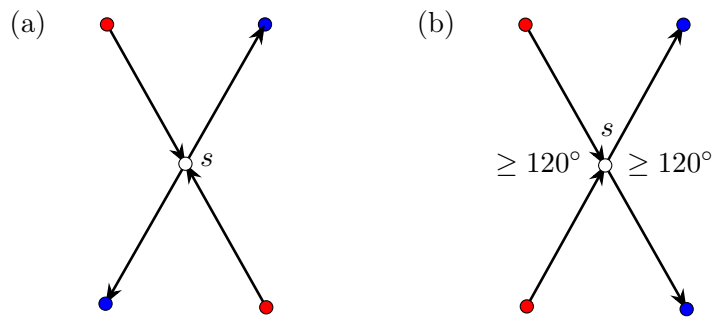

\item\label{23} $\deg(s)=(2,3), (3,2)$.

In the case $(2,3)$, the three outgoing directed edges are pairwise at $120^\circ$ degrees, and the two incoming directed edges lie on a straight line (Fig.~\ref{fig23}(a)).
The case $(3,2)$ is exactly opposite: the three incoming directed edges are pairwise at $120^\circ$ degrees, and the two outgoing directed edges lie on a straight line (Fig.~\ref{fig23}(b)).

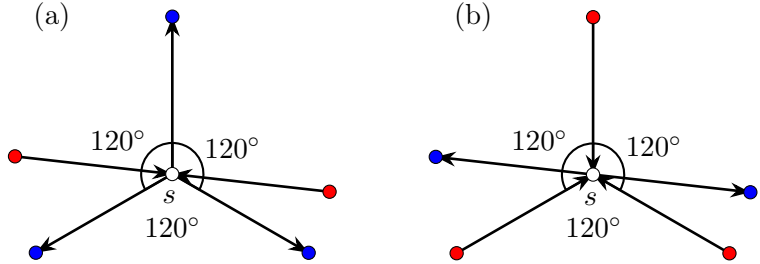
\begin{figure}[h]
\definecolor{qqqqff}{rgb}{0.,0.,1.}
\definecolor{ffffff}{rgb}{1.,1.,1.}
\definecolor{ffqqqq}{rgb}{1.,0.,0.}
\centering
\begin{tikzpicture}[line cap=round,line join=round,>=Stealth,x=0.4cm,y=0.4cm]
\draw [shift={(1.9982402614052646,-1.2350546258915043)},line width=0.8pt] (0,0) -- (89.98074029825241:1.0346598174879706) arc (89.98074029825241:209.98074029825236:1.0346598174879706) -- cycle;
\draw [shift={(1.9982402614052646,-1.2350546258915043)},line width=0.8pt] (0,0) -- (-30.019259701747604:1.0346598174879706) arc (-30.019259701747604:89.98074029825243:1.0346598174879706) -- cycle;
\draw [shift={(15.983392127784313,-1.2522989561829603)},line width=0.8pt] (0,0) -- (89.98074029825217:1.0346598174879706) arc (89.98074029825217:209.98074029825216:1.0346598174879706) -- cycle;
\draw [shift={(15.983392127784313,-1.2522989561829603)},line width=0.8pt] (0,0) -- (-30.01925970174783:1.0346598174879706) arc (-30.01925970174783:89.98074029825219:1.0346598174879706) -- cycle;
\draw [<-,line width=1pt] (2.,4.) -- (1.9982402614052646,-1.2350546258915038);
\draw [<-,line width=1pt] (-2.536329904113388,-3.8510579605101936) -- (1.9982402614052646,-1.2350546258915043);
\draw [<-,line width=1pt] (1.9982402614052646,-1.2350546258915043) -- (-3.2292402566406757,-0.6466604829516096);
\draw [<-,line width=1pt] (1.9982402614052646,-1.2350546258915043) -- (7.225720779451205,-1.823448768831399);
\draw [<-,line width=1pt] (6.53105068832918,-3.854105917164317) -- (1.9982402614052646,-1.2350546258915043);
\draw [<-,line width=1pt] (15.983392127784313,-1.2522989561829603) -- (15.98515186637907,3.982755669708533);
\draw [<-,line width=1pt] (21.210872645830236,-1.8406930991228447) -- (15.983392127784313,-1.2522989561829603);
\draw [<-,line width=1pt] (10.755911609738392,-0.6639048132430758) -- (15.983392127784313,-1.2522989561829603);
\draw [<-,line width=1pt] (15.983392127784313,-1.2522989561829603) -- (11.44882196226566,-3.8683022908016254);
\draw [<-,line width=1pt] (15.983392127784313,-1.2522989561829603) -- (20.516202554708208,-3.8713502474557853);
\draw [fill=qqqqff] (6.53105068832918,-3.854105917164317) circle (2.5pt);
\draw (-2,4) node {(a)};
\draw [fill=qqqqff] (2.,4.) circle (2.5pt);
\draw [fill=qqqqff] (-2.536329904113388,-3.8510579605101936) circle (2.5pt);
\draw[color=black] (0.2,-0.1120862439161584) node {$120^\circ$};
\draw[color=black] (2,-3) node {$120^\circ$};
\draw [fill=ffqqqq] (-3.2292402566406757,-0.6466604829516096) circle (2.5pt);
\draw [fill=ffqqqq] (7.225720779451205,-1.823448768831399) circle (2.5pt);
\draw[color=black] (4,-0.35) node {$120^\circ$};
\draw [fill=ffqqqq] (20.516202554708208,-3.8713502474557857) circle (2.5pt);
\draw (12,4) node {(b)};
\draw [fill=ffqqqq] (15.98515186637907,3.9827556697085327) circle (2.5pt);
\draw [fill=ffqqqq] (11.44882196226566,-3.8683022908016254) circle (2.5pt);
\draw[color=black] (14.2,-0.1120862439161584) node {$120^\circ$};
\draw[color=black] (16,-3) node {$120^\circ$};
\draw [fill=qqqqff] (10.755911609738392,-0.6639048132430758) circle (2.5pt);
\draw [fill=qqqqff] (21.210872645830236,-1.8406930991228447) circle (2.5pt);
\draw[color=black] (18,-0.3) node {$120^\circ$};
\draw [fill=ffffff] (1.9982402614052646,-1.2350546258915043) circle (2.5pt);
\draw[color=black] (1.9,-2) node {$s$};
\draw [fill=ffffff] (15.983392127784313,-1.2522989561829603) circle (2.5pt);
\draw[color=black] (15.9,-2.05) node {$s$};
\end{tikzpicture}
\caption{$\deg(s)=(2,3)$ or $(3,2)$}\label{fig23}
\end{figure}

\item\label{33} $\deg(s)=(3,3)$.
The incoming and outgoing edges alternate, with consecutive directed edges at $60^\circ$ angles (Fig.~\ref{fig33}(a)).
\begin{figure}[h]
\definecolor{qqqqff}{rgb}{0.,0.,1.}
\definecolor{ffffff}{rgb}{1.,1.,1.}
\definecolor{ffqqqq}{rgb}{1.,0.,0.}
\centering
\begin{tikzpicture}[line cap=round,line join=round,>=Stealth,x=0.4cm,y=0.4cm]
\draw (-3,4) node {(a)};
\draw [->,line width=1pt] (2.,4.) -- (1.9982402614052646,-1.2350546258915038);
\draw [->,line width=1pt] (-2.536329904113388,-3.8510579605101936) -- (1.9982402614052646,-1.2350546258915043);
\draw [->,line width=1pt] (6.53105068832918,-3.854105917164317) -- (1.9982402614052646,-1.2350546258915043);
\draw [->,line width=1pt] (1.9982402614052646,-1.2350546258915043) -- (6.532810426923917,1.380948708727185);
\draw [->,line width=1pt] (1.9982402614052646,-1.2350546258915043) -- (1.9964805228105291,-6.470109251783008);
\draw [->,line width=1pt] (1.9982402614052646,-1.2350546258915043) -- (-2.5345701655186508,1.3839966653813085);
\draw [fill=ffffff] (1.9982402614052646,-1.2350546258915043) circle (2.5pt);
\draw[color=black] (1.5,-0.3) node {$s$};
\draw [fill=ffqqqq] (6.53105068832918,-3.854105917164317) circle (2.5pt);
\draw [fill=ffqqqq] (2.,4.) circle (2.5pt);
\draw [fill=ffqqqq] (-2.536329904113388,-3.8510579605101936) circle (2.5pt);
\draw [fill=qqqqff] (1.9964805228105291,-6.4701092517830086) circle (2.5pt);
\draw [fill=qqqqff] (6.532810426923917,1.380948708727185) circle (2.5pt);
\draw [fill=qqqqff] (-2.5345701655186508,1.3839966653813085) circle (2.5pt);
\begin{scope}[xshift=6cm]
\draw (-3,4) node {(b)};
\draw [->,line width=1pt] (2.,4.) -- (6.532810426923917,1.380948708727185);
\draw [->,line width=1pt] (6.532810426923917,1.380948708727185) -- (6.53105068832918,-3.854105917164317);
\draw [->,line width=1pt] (6.53105068832918,-3.854105917164317)-- (1.9964805228105291,-6.470109251783008);
\draw [->,line width=1pt] (1.9964805228105291,-6.470109251783008) -- (-2.536329904113388,-3.8510579605101936);
\draw [->,line width=1pt] (-2.536329904113388,-3.8510579605101936) -- (-2.5345701655186508,1.3839966653813085);
\draw [->,line width=1pt] (-2.5345701655186508,1.3839966653813085) -- (2.,4.);
\draw [fill=ffqqqq] (6.53105068832918,-3.854105917164317) circle (2.5pt);
\draw [fill=ffqqqq] (2.,4.) circle (2.5pt);
\draw [fill=ffqqqq] (-2.536329904113388,-3.8510579605101936) circle (2.5pt);
\draw [fill=qqqqff] (1.9964805228105291,-6.4701092517830086) circle (2.5pt);
\draw [fill=qqqqff] (6.532810426923917,1.380948708727185) circle (2.5pt);
\draw [fill=qqqqff] (-2.5345701655186508,1.3839966653813085) circle (2.5pt);
\end{scope}
\end{tikzpicture}
\caption{Sources and sinks are alternating vertices of a regular hexagon. Since the directed network in (a) with $\deg(s)=(3,3)$ is shortest, the directed network in (b) is also shortest and contains a cycle \cite{Alfaro-Thesis}}\label{fig33}
\end{figure}
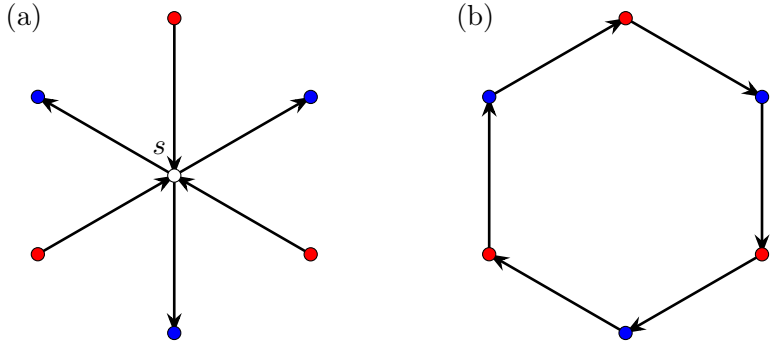

\end{enumerate}
\end{theorem}

Case \ref{12} in the above theorem was well known and follows easily from classical results in elementary geometry (see Lemmas~\ref{lem:FT} and \ref{lem:FT2}).
Case \ref{22}\ref{b} was known \cite[Theorem~2.4]{ACSS} and we do not prove it in this paper, but Case \ref{22}\ref{a} is new, and is proved as Theorem~\ref{thm4} in Section~\ref{sec:deg4}.
We present a proof that uses a different norm on $\bR^2$ and that gives as a byproduct results for the $\ell_1$-plane (Corollaries~\ref{cor1} and \ref{cor2} below).

Case \ref{23} is new and the hardest part of the theorem.
We devote Section~\ref{sec:deg5} to its proof.
It was known in the special case where the two incoming directed edges (where $\deg(s)=(2,3)$) are orthogonal to one of the outgoing directed edges \cite{Alfaro-Thesis}.
Case \ref{33} was known \cite{Alfaro-Thesis}, but we give a different proof in Section~\ref{sec:deg6}, this time using the norm with the regular hexagon as unit ball (Theorem~\ref{thm6}), and again with a corollary for shortest direct networks in this norm (Corollary~\ref{cor3}).

An immediate consequence of Theorem~\ref{thm:main} is that shortest directed networks can be constructed with straightedge and compass.
\begin{corollary}\label{cor:main}
Given two finite sets $A$ and $B$ of points in the Euclidean plane, any shortest $(A,B)$-network can be constructed from $A$ and $B$ by straightedge and compass.
\end{corollary}
\begin{proof}
Consider the underlying graph $G$ of a shortest $(A,B)$-network.
Since we can easily construct a Steiner point of degree $(2,2)$ from its neighbours, we may without loss of generality replace each Steiner point of degree $4$ and its incident edges by two edges joining opposite neighbours.
We can similarly replace the incoming (outgoing) edges of a Steiner point of degree $(2,3)$ (degree $(3,2)$, respectively) by an edge joining the two neighbours.
What remains are Steiner points with three neighbours joined by edges that are pairwise at $120^\circ$.
It is then clear that we can decompose $G$ into a union of full Steiner trees (trees in which each non-Steiner point has degree $1$) such that any two full Steiner trees intersect in only finitely many points.
Each of these trees is constructible by the Melzak--Hwang algorithm \cite[Section~1.2.1]{BZ}.
\end{proof}
\begin{corollary}\label{cor:algorithm}
Given finite sets $A$ and $B$ of points with rational coordinates in the Euclidean plane, there is an algorithm that constructs a shortest $(A,B)$-network.
\end{corollary}
We leave open the following problems.
\begin{problem}
Find an algorithm that can feasibly compute shortest directed networks in the Euclidean plane, at least for a small number of sources and sinks.
\end{problem}
The main difficulty in finding an algorithm lies in the enumeration of all possible digraph structures.
Note that if there is only one source or one sink, then a shortest network has to be a minimum Steiner tree, which can be computed with the GeoSteiner algorithm \cite[Section~1.4]{BZ}.

\begin{problem}
Find a characterization of the local structure of sources or sinks of shortest directed networks in the Euclidean plane analogous to that for Steiner points in Theorem~\ref{thm:main}.
\end{problem}
Note that such a characterization is known in the undirected case: In a Steiner minimal tree, a given node either has degree $3$ with all angles between incident edges equal to $120^\circ$, or has degree $2$ with the angle between the two edges $\geq 120^\circ$, or has degree $1$ \cite[Theorem~1.2]{BZ}.

\begin{problem}
Find characterizations of the local structure of nodes (Steiner points, sources or sinks) of shortest directed networks in higher-dimensional Euclidean space analogous to that for Steiner points in Theorem~\ref{thm:main}.
\end{problem}
Corollary~\ref{cor:FT3} below gives some partial results on the above two problems.

\begin{problem}\label{prob}
Find characterizations of the local structure of nodes (Steiner points, sources or sinks) of shortest directed networks in other normed planes and spaces.
\end{problem}
The undirected case of Problem~\ref{prob}, namely to characterize the local structure of Steiner points and terminals in Steiner minimal trees in normed planes, is known \cite{SwaNetworks}.

We also draw attention to the following attractive conjecture of Alfaro.
\begin{conjecture}[Alfaro \cite{Alfaro-Thesis}]
In the Euclidean plane, suppose that the set $A$ of sources and set $B$ of sinks are the same.
Then a shortest $(A,B)$-network does not have Steiner points and is a union of cycles.
\end{conjecture}

\section{Basic lemmas}\label{sec:basic}
By a \emph{norm} defined on $\bR^d$, we mean a function $\norm{\cdot}\colon\bR^d\to\bR$ such that $\norm{x}\geq 0$ for all $x\in\bR^d$, $\norm{\lambda x}=\abs{\lambda}\norm{x}$ for all $\lambda\in\bR$ and $x\in\bR^d$, and $\norm{x+y}\leq\norm{x}+\norm{y}$ for all $x,y\in\bR^d$.
The unit ball of a norm is defined to be the set $\setbuilder{x\in\bR^d}{\norm{x}\leq 1}$.
We will compare norms, and to do this, it is useful to keep in mind that $\norm{x}_a\geq\norm{x}_b$ for all $x\in\bR^d$ if and only if the corresponding unit balls $B_a\subseteq B_b$.

We denote the $d$-dimensional Euclidean space by $\bE^d$, that is, the $d$-dimensional normed space with norm \[\norm{(x_1,\dots,x_d)}_2 = \sqrt{x_1^2+\dots+x_d^2}.\]
We will work with a variety of norms in the plane $\bR^2$ apart from the Euclidean norm $\norm{\cdot}_2$ with unit ball $B_2$.
The $\ell_1$-norm is defined by $\norm{\opair{x}{y}}_1=|x|+|y|$.
Its unit ball $B_1$ is the convex hull of the four points $\pm e_1$ and $\pm e_2$, where $e_1=\opair{1}{0}$ and $e_2=\opair{0}{1}$ form the standard unit basis of $\bR^2$.
Since $B_1\subseteq B_2$, it follows that $\norm{x}_2\leq\norm{x}_1$ for all $x\in\bR^2$.
Alternatively,
\[ \norm{\opair{x}{y}}_2^2 = x^2+y^2 \leq \abs{x}^2 + 2\abs{x}\abs{y}+\abs{y}^2 = (\abs{x}+\abs{y})^2 = \norm{\opair{x}{y}}_1^2.\]
For any $\theta\in(0,90^\circ)$ we define the norm $\norm{\cdot}_{1(\theta)}$ on $\bR^2$ by \[\norm{\opair{x}{y}}_{1(\theta)}=\abs{x}\cos\theta+\abs{y}\sin\theta.\]
\begin{lemma}\label{lemma1}
For any $\theta\in(0^\circ,90^\circ)$ and $\opair{x}{y}\in\bR^2$, $\norm{\opair{x}{y}}_2\geq\norm{\opair{x}{y}}_{1(\theta)}$.
\end{lemma}
\begin{proof}
By the Cauchy--Schwarz inequality,
\begin{align*}
\norm{\opair{x}{y}}_{1(\theta)} &= \ipr{\opair{\abs{x}}{\abs{y}}}{\opair{\cos\theta}{\sin\theta}}\\
&\leq \norm{\opair{\abs{x}}{\abs{y}}}_2\norm{\opair{\cos\theta}{\sin\theta}}_2 = \norm{\opair{x}{y}}_2.\qedhere
\end{align*}
\end{proof}
The inequality in Lemma~\ref{lemma1} is sharp with equality attained at the points $\opair{\pm\cos\theta}{\pm\sin\theta}$.
This lemma can also be seen by noting that the unit ball $B_{1(\theta)}$ of $\norm{\cdot}_{1(\theta)}$ is a parallelogram 
circumscribing the Euclidean unit ball $B_2$, and touching $B_2$ at the four points $\opair{\pm\cos\theta}{\pm\sin\theta}$.
See Fig.~\ref{fig3} for the case $\theta=60^\circ$, where the four points are labelled $a_2$, $a_3$, $b_2$, $b_3$.
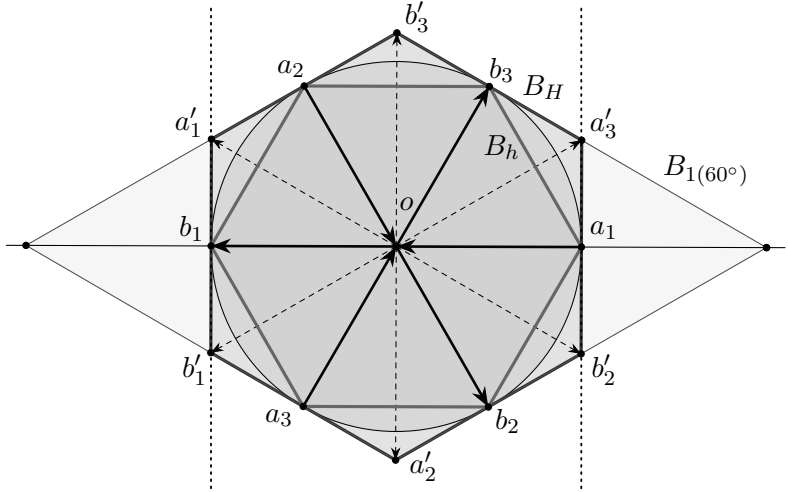
\begin{figure}
\centering
\definecolor{ffqqqq}{rgb}{0.3,0.3,0.3}
\definecolor{fffqqq}{rgb}{0.4,0.4,0.4}
\definecolor{zzttqq}{rgb}{0.2,0.2,0.2}
\begin{tikzpicture}[line cap=round,line join=round,>=Stealth,x=1.0cm,y=1.0cm,scale=0.8]
\fill[line width=1.2pt,color=zzttqq,fill=zzttqq,fill opacity=0.1] (7.931247533212796,2.4145072694285403) -- (4.86,4.2) -- (1.7780941704462743,2.432967980213158) -- (1.7674358741053435,-1.1195567701451448) -- (4.838683407318138,-2.905049500716605) -- (7.920589236871864,-1.1380174809297645) -- cycle;
\draw [fill=black,fill opacity=0.05] (4.84934170365907,0.6474752496416988) circle (3.0765905278232757cm);
\fill[color=ffqqqq,fill=ffqqqq,fill opacity=0.05] (-1.3038116591074507,0.665935960426316) -- (4.86,4.2) -- (11.002495066425595,0.6290145388570779) -- (4.838683407318138,-2.905049500716605) -- cycle;
\draw [line width=1.2pt,color=zzttqq] (7.931247533212796,2.4145072694285403)-- (4.86,4.2);
\draw [line width=1.2pt,color=zzttqq] (4.86,4.2)-- (1.7780941704462743,2.432967980213158);
\draw [line width=1.2pt,color=zzttqq] (1.7780941704462743,2.432967980213158)-- (1.7674358741053435,-1.1195567701451448);
\draw [line width=1.2pt,color=zzttqq] (1.7674358741053435,-1.1195567701451448)-- (4.838683407318138,-2.905049500716605);
\draw [line width=1.2pt,color=zzttqq] (4.838683407318138,-2.905049500716605)-- (7.920589236871864,-1.1380174809297645);
\draw [line width=1.2pt,color=zzttqq] (7.920589236871864,-1.1380174809297645)-- (7.931247533212796,2.4145072694285403);
\draw [domain=-1.6313921210724558:11.310760314574766] plot(\x,{(--4.073536804327942-0.01846071078461864*\x)/6.1531533627665205});
\draw [->,dash pattern=on 2pt off 2pt] (4.849341703659069,0.6474752496416972) -- (7.931247533212796,2.4145072694285403);
\draw [->,dash pattern=on 2pt off 2pt] (4.849341703659069,0.6474752496416972) -- (4.86,4.2);
\draw [->,dash pattern=on 2pt off 2pt] (4.849341703659069,0.6474752496416972) -- (1.7780941704462743,2.432967980213158);
\draw [->,dash pattern=on 2pt off 2pt] (4.849341703659069,0.6474752496416972) -- (1.7674358741053435,-1.1195567701451448);
\draw [->,dash pattern=on 2pt off 2pt] (4.849341703659069,0.6474752496416972) -- (4.838683407318138,-2.905049500716605);
\draw [->,dash pattern=on 2pt off 2pt] (4.849341703659069,0.6474752496416972) -- (7.920589236871864,-1.1380174809297645);
\draw [->,line width=1pt] (3.3190470852231373,3.316483990106579) -- (4.849341703659069,0.6474752496416971);
\draw [->,line width=1pt] (4.849341703659069,0.6474752496416972) -- (6.379636322095001,-2.0215334908231846);
\draw [->,line width=1pt] (7.92591838504233,0.6382448942493879) -- (4.849341703659069,0.6474752496416972);
\draw [->,line width=1pt] (4.849341703659069,0.6474752496416972) -- (1.772765022275809,0.6567056050340065);
\draw [->,line width=1pt] (3.3030596407117407,-2.0123031354308747) -- (4.849341703659069,0.6474752496416971);
\draw [->,line width=1pt] (4.849341703659069,0.6474752496416972) -- (6.395623766606398,3.3072536347142703);
\draw [color=ffqqqq] (-1.3038116591074507,0.665935960426316)-- (4.86,4.2);
\draw [color=ffqqqq] (4.86,4.2)-- (11.002495066425595,0.6290145388570779);
\draw [color=ffqqqq] (11.002495066425595,0.6290145388570779)-- (4.838683407318138,-2.905049500716605);
\draw [color=ffqqqq] (4.838683407318138,-2.905049500716605)-- (-1.3038116591074507,0.665935960426316);
\draw [dotted,line width=0.7pt] (7.924003520629264,-3.3734828343302645) -- (7.924003520629264,4.630641365271997);
\draw [dotted,line width=0.7pt] (1.7707947719780461,-3.3734828343302645) -- (1.7707947719780461,4.630641365271997);
\draw[line width=1.2pt,color=fffqqq,fill=fffqqq,fill opacity=0.05] (7.92591838504233,0.6382448942493879) -- (6.395623766606398,3.3072536347142703) -- (3.3190470852231373,3.316483990106579) -- (1.772765022275809,0.6567056050340065) -- (3.3030596407117407,-2.0123031354308747) -- (6.379636322095001,-2.0215334908231846) -- cycle;
\draw [fill=black] (7.931247533212796,2.4145072694285403) circle (1.5pt);
\draw[color=black] (8.3,2.6941597040779017) node {$a_3'$};
\draw [fill=black] (4.86,4.2) circle (1.5pt);
\draw[color=black] (5.2,4.469267893505823) node {$b_3'$};
\draw [fill=black] (1.7780941704462743,2.432967980213158) circle (1.5pt);
\draw[color=black] (1.4,2.710297051254519) node {$a_1'$};
\draw [fill=black] (1.7674358741053435,-1.1195567701451448) circle (1.5pt);
\draw[color=black] (1.5,-1.4) node {$b_1'$};
\draw [fill=black] (4.838683407318138,-2.905049500716605) circle (1.5pt);
\draw[color=black] (5.3,-3) node {$a_2'$};
\draw [fill=black] (7.920589236871864,-1.1380174809297645) circle (1.5pt);
\draw[color=black] (8.3,-1.4) node {$b_2'$};
\draw [fill=black] (6.395623766606398,3.3072536347142703) circle (1.5pt);
\draw[color=black] (6.6,3.581713798791862) node {$b_3$};
\draw[color=black] (7.3,3.3) node {$B_H$};
\draw [fill=black] (3.3190470852231373,3.316483990106579) circle (1.5pt);
\draw[color=black] (3.1,3.5978511459684794) node {$a_2$};
\draw [fill=black] (1.772765022275809,0.6567056050340065) circle (1.5pt);
\draw[color=black] (1.45,0.95) node {$b_1$};
\draw [fill=black] (3.3030596407117407,-2.0123031354308747) circle (1.5pt);
\draw[color=black] (2.9,-2.2) node {$a_3$};
\draw [fill=black] (6.379636322095001,-2.0215334908231846) circle (1.5pt);
\draw[color=black] (6.7,-2.25) node {$b_2$};
\draw [fill=black] (7.92591838504233,0.6382448942493879) circle (1.5pt);
\draw[color=black] (8.3,0.92) node {$a_1$};
\draw[color=black] (10,1.9) node {$B_{1(60^\circ)}$};
\draw[color=black] (6.6,2.3) node {$B_{h}$};
\draw [fill=black] (4.849341703659069,0.6474752496416972) circle (1.5pt);
\draw[color=black] (5.02,1.35) node {$o$};
\draw [fill=black] (-1.3038116591074507,0.665935960426316) circle (1.5pt);
\draw [fill=black] (11.002495066425595,0.6290145388570779) circle (1.5pt);
\draw [fill=black] (11.002495066425595,0.6290145388570787) circle (1.5pt);
\end{tikzpicture}
\caption{The unit balls of $\norm{\cdot}_h$, $\norm{\cdot}_2$, $\norm{\cdot}_H$ and $\norm{\cdot}_{1(60^\circ)}$}\label{fig3}
\end{figure}

Define the norm \[\norm{\opair{x}{y}}_H = \max\set{\abs{x},\tfrac12\abs{x}+\tfrac{\sqrt{3}}{2}\abs{y}} = \max\set{\abs{x}, \norm{\opair{x}{y}}_{1(60^\circ)}}.\]
Its unit ball $B_H$ is a regular hexagon $a_1'b_1'a_2'b_2'a_3'b_3'$ with vertices $a_1'=-b_2'=\opair{-1}{1/\sqrt{3}}$, $a_2'=-b_3'=\opair{0}{-2/\sqrt{3}}$, $a_3'=-b_1'=\opair{1}{1/\sqrt{3}}$, circumscribing the Euclidean unit circle and touching it at $a_1,a_2,a_3,b_1,b_2,b_3$ (Fig.~\ref{fig3}), where $a_1=-b_1=\opair{1}{0}$, $a_2=-b_2=\opair{-1/2}{\sqrt{3}/2}$, $a_3=-b_3=\opair{-1/2}{-\sqrt{3}/2}$.
Also define the norm
\[\norm{\opair{x}{y}}_h = \max\set{\tfrac{2}{\sqrt{3}}\abs{y},\abs{x}+\tfrac{1}{\sqrt{3}}\abs{y}}. \]
Its unit ball $B_h$ is the regular hexagon $a_1b_3a_2b_1a_3b_2$ inscribed in the Euclidean unit circle (Fig.~\ref{fig3}).
Note that although the normed plane $(\bR^2,\norm{\cdot}_h)$ is not the same as $(\bR^2,\norm{\cdot}_H)$, they are isometric.
\begin{lemma}\label{lemma2}
For any $\opair{x}{y}\in\bR^2$, $\norm{\opair{x}{y}}_h\geq\norm{\opair{x}{y}}_2\geq\norm{\opair{x}{y}}_H$.
\end{lemma}
\begin{proof}
For the first inequality, note that if $3x^2\leq y^2$ then $x^2+y^2\leq \frac{4}{3}y^2$, while if $3x^2\geq y^2$, then $x^2+y^2\leq (\abs{x}+\frac{1}{\sqrt{3}}\abs{y})^2$.
For the second inequality, note that by Lemma~\ref{lemma1}, $\norm{\opair{x}{y}}\geq\norm{\opair{x}{y}}_{1(60^\circ)}$, and trivially, $\norm{\opair{x}{y}}=\sqrt{x^2+y^2}\geq\abs{x}$.
\end{proof}
This lemma also follows from the fact that the unit balls $B_h\subseteq B_2\subseteq B_H$.
\begin{lemma}\label{lemma4}
For any $a,b\in\bR^2$, if $a'$ and $b'$ are the orthogonal projections of $a$ and $b$ onto the $x$-axis, then $\norm{a-b}_H\geq\norm{a'-b'}_H=\norm{a'-b'}_2$.
\end{lemma}
\begin{proof}
Note that $\norm{\opair{x}{y}}_H\geq\abs{x}=\norm{\opair{x}{0}}_H=\norm{\opair{x}{0}}_2$.
\end{proof}
\begin{lemma}\label{lemma3}
Suppose that the unit ball $B$ of the norm $\norm{\cdot}$ on $\bR^2$ is a polygon with edges $p_ip_{i+1}$, $i=1,\dots,2n-1$, where $p_{n+i}=-p_i$, $i=1,\dots, n$ and $p_{2n}=p_0$.
Then for any segment $vw$ in the plane there exists a point $c$ such that the segments $vc$ and $cw$ are each parallel to $p_i$ and $p_{i+1}$ for some $i=1,\dots,n-1$ and $\norm{v-w}=\norm{v-c}+\norm{c-w}$.
\end{lemma}
\begin{proof}
Let $i=1,\dots,2n-1$ be such that the unit vector $\norm{w-v}^{-1}(w-v)$ lies on the edge $p_ip_{i+1}$ of $B$.
Then $w-v = \alpha p_i + \beta p_{i+1}$ for some $\alpha,\beta\geq 0$.
It is then easy to see that $c=w-\beta p_{i+1} = v+\alpha p_i$ is the required point.
\end{proof}
We call the union of the segments $vc\cup cw$ in the above lemma a \emph{broken segment} that has the same length as $vw$.
Likewise, we call the directed path consisting of the edges $\arc{v}{c}$ and $\arc{c}{w}$ a \emph{broken edge}.
Note that if we replace the directed edge $\arc{v}{w}$ in an $(A,B)$-network by the broken edge $\arc{v}{c}$, $\arc{c}{w}$, then we still have an $(A,B)$-network of the same length.

We next note that only the directions of the directed edges incident to a vertex are important in characterizing the local structure of a node.
\begin{lemma}\label{lem:local}
Let $A=\{a_1,\dots,a_m\}$ and $B=\{b_1,\dots,b_n\}$ be sources and sinks in a $d$-dimensional normed space $X$.
Suppose that the $(A,B)$-network $G$ with directed edge set $E=\setbuilder{\arc{a_i}{o}}{a_i\neq o, i=1,\dots,m}\cup\setbuilder{\arc{o}{b_j}}{b_j\neq o, j=1,\dots,n}$ is a shortest $(A,B)$-network.
Let $a_i'$ be any point on the ray from $o$ to $a_i$ and $b_j'$ any point on the ray from $o$ to $b_j$.
Let $A'=\{a_1',\dots,a_m'\}$ and $B'=\{b_1',\dots,b_n'\}$.
Then the network $G'$ with directed edge set $E'= \setbuilder{\arc{a_i'}{o}}{a_i'\neq o, i=1,\dots,m}\cup\setbuilder{\arc{o}{b_j'}}{b_j'\neq o, j=1,\dots,n}$ is a shortest $(A',B')$-network.
\end{lemma}
\begin{proof}
By scaling, we may assume that each $\norm{a_i'}\leq\norm{a_i}$ and each $\norm{b_j'}\leq\norm{b_j}$.
If $G'$ is not a shortest $(A,B)$-network, then $G$ can be shortened by replacing the part of $G$ that coincides with $G'$ by a shorter network, which gives a contradiction.
\end{proof}

\begin{lemma}\label{lem:convex}
All vertices and directed edges of a shortest $(A,B)$-network in Euclidean space $\bE^d$ are contained in the convex hull of $A \cup B$.
\end{lemma}
\begin{proof}
It is sufficient to prove that each Steiner point of the $(A,B)$-network $G$ is contained in $\conv(A\cup B)$.
Suppose that some Steiner point $s\notin\conv(A\cup B)$.
Let $H$ be a hyperplane that strictly separates $s$ from all other nodes of $G$.
Each directed edge $e$ incident to $s$ intersects $H$ in some point $p_e$.
If we project $s$, together with the part $e'$ joining $s$ and $p_e$ of each edge $e$ orthogonally onto $H$, then we obtain a shorter network, which is a contradiction.
\end{proof}

The following is a well-known geometric result that goes back to Fermat and Torricelli \cite[Problem 91]{Do}.
\begin{lemma}\label{lem:FT}
Let $a$, $b$, $c$ be three points in Euclidean space $\bE^d$.
Then there is a unique point $s$ that minimizes the sum of distances $\norm{a-s}_2 + \norm{b-s}_2 + \norm{c-s}_2$ to the given points.
If $\myangle abc, \myangle bca, \myangle cab < 120^\circ$, then $s$ is in the relative interior of $\conv\{a,b,c\}$, and $\myangle asb = \myangle bsc = \myangle csa = 120^\circ$.
If, on the other hand, $\myangle abc \geq 120^\circ$, say, then $s=b$.
\end{lemma}

A proof of the following lemma can be found in \cite[p.~22]{Gilbert-Pollak}.
\begin{lemma}\label{lem:FT2}
Let $a,b,c$ be three points in Euclidean space $\bE^d$ distinct from the origin $o$.
Suppose that all three angles $\myangle aob, \myangle boc, \myangle coa\geq 120^\circ$.
Then $a,b,c,o$ lie in the same $2$-dimensional plane, and $\myangle aob = \myangle boc = \myangle coa = 120^\circ$.
\end{lemma}
As a consequence of the previous two lemmas, we obtain the following properties of shortest directed networks in Euclidean space.
\begin{lemma}\label{lem:FT3}
Let $G=(V,E)$ be a shortest $(A,B)$-network in Euclidean space $\bE^d$.
\begin{enumerate}
\item\label{lem:FT3:1} Let $v\in V$.
Then the angle at $v$ between any two incoming directed edges or any two outgoing directed edges at $v$ is $\geq 120^\circ$.
Consequently, $\deg^-(v), \deg^+(v)\leq 3$.
\item\label{lem:FT3:2} Let $v\in V\setminus B$ satisfy $\deg^+(v)=1$.
Then the angle at $v$ between the outgoing directed edge and any incoming directed edge is $\geq 120^\circ$.
Consequently, $\deg^-(v)\leq 2$.
\item\label{lem:FT3:3} Let $v\in V\setminus B$ satisfy $\deg^-(v)=1$.
Then the angle at $v$ between the incoming directed edge and any outgoing directed edge is $\geq 120^\circ$.
Consequently, $\deg^+(v)\leq 2$.
\end{enumerate}
\end{lemma}
\begin{proof}
\ref{lem:FT3:1}.
By Lemma~\ref{lem:FT}, if some angle between two incoming directed edges is $<120^\circ$, then we can replace these two directed edges by three directed edges joined to a Steiner point, oriented appropriately, to obtain a shorter $(A,B)$-network.
It then follows from Lemma~\ref{lem:FT2} that there cannot be more than three incoming edges.

\smallskip
\ref{lem:FT3:2}.
Let the outgoing directed edge be $\arc{v}{b}$ and consider any incoming directed edge $\arc{c}{v}$ with $\myangle bvc<120^\circ$.
Note that by minimality of $G$, since $v$ is not a sink and has only one outgoing directed edge, any directed path from a source to a sink that uses $\arc{c}{v}$ also has to use $\arc{v}{b}$.
We can then replace $\arc{v}{b}$ and $\arc{c}{v}$ by a Steiner point $s$ from Lemma~\ref{lem:FT} and directed edges $\arc{c}{s}$, $\arc{v}{s}$ and $\arc{s}{b}$, to obtain a shorter $(A,B)$-network.

\smallskip
\ref{lem:FT3:3}. Similar to \ref{lem:FT3:2} by changing the directions of all directed edges.
\end{proof}
\begin{corollary}[{\cite[Theorem~3.1]{Alfaro-Thesis}}]\label{cor:FT3}
The only possible degrees of a node in a simple $(A,B)$-network in any Euclidean space of dimension at least $2$ are $(1,2)$, $(2,1)$, $(2,2)$, $(2,3)$, $(3,2)$ and $(3,3)$.
\end{corollary}

\section{Steiner points of degree 4}\label{sec:deg4}
\begin{theorem}\label{thm4}
Let $a_1b_1a_2b_2$ be a convex quadrilateral in the Euclidean plane with diagonals $a_1a_2$ and $b_1b_2$ intersecting in $o$.
Let $A=\{a_1,a_2\}$ be the set of sources and $B=\{b_1,b_2\}$ the set of sinks.
Then the network with directed edges $\arc{a_i}{s}$ and $\arc{s}{b_i}$, $i=1,2$, is a shortest $(A,B)$-network if, and only if $s=o$.
\end{theorem}
\begin{proof}
We first note that by the triangle inequality it follows that if the network with edge set $\set{\arc{a_1}{s},\arc{a_2}{s},\arc{s}{b_1},\arc{s}{b_2}}$ is shortest, then $s=o$.

We next show the converse.
By Lemma~\ref{lem:local}, it is sufficient to prove the theorem with $a_1b_1a_2b_2$ a rectangle.
From now on we assume that $a_1b_1a_2b_2$ is a rectangle with half diagonals of length $\norm{a_1}_2=\norm{a_2}_2=\norm{b_1}_2=\norm{b_2}_2=1$.
We choose coordinates such that $o$ is the origin, $a_1$ is in the first quadrant, and the axes bisect the four angles created by the diagonals $a_1a_2$ and $b_1b_2$ at $o$.
Let $\theta$ be the angle that $oa_1$ makes with the positive $x$-axis.
(See Fig.~\ref{fig2}.)
\begin{figure}
\centering
\begin{tikzpicture}[line cap=round,line join=round,>=Stealth,scale=0.5]
\draw [shift={(5.5873202253576,0.93376402396454)}] (0,0) -- (-0.12229057161031615:0.8) arc (-0.12229057161031615:40.510780060725196:0.8) -- cycle;
\draw(1.1211305251330388,4.357764000364571) -- (1.5453936274675397,4.3568584632285985) -- (1.5462991646035116,4.7811215655631) -- (1.1220360622690104,4.782027102699072) -- cycle; 
\draw(9.644727811183497,4.763836470908054) -- (9.643822274047526,4.339573368573554) -- (10.068085376382026,4.338667831437582) -- (10.068990913517998,4.762930933772083) -- cycle; 
\draw(10.053509925582162,-2.49023595243549) -- (9.62924682324766,-2.489330415299518) -- (9.628341286111688,-2.9135935176340193) -- (10.05260438844619,-2.914499054769991) -- cycle; 
\draw(1.5299126395317022,-2.8963084229789757) -- (1.5308181766676738,-2.4720453206444746) -- (1.106555074333173,-2.471139783508503) -- (1.1056495371972013,-2.895402885843004) -- cycle; 
\draw (1.1220360622690104,4.782027102699072)-- (10.068990913517998,4.762930933772083);
\draw (10.068990913517998,4.762930933772083)-- (10.05260438844619,-2.914499054769991);
\draw (10.05260438844619,-2.914499054769991)-- (1.1056495371972013,-2.895402885843004);
\draw (1.1056495371972013,-2.895402885843004)-- (1.1220360622690104,4.782027102699072);
\draw (1.1056495371972013,-2.895402885843004)-- (10.068990913517998,4.762930933772083);
\draw (1.1220360622690104,4.782027102699072)-- (10.05260438844619,-2.914499054769991);
\draw [->] (0.1600114145243677,0.9453479435878902) -- (11.239960939122417,0.9216991612544564);
\draw [->] (5.578113129109964,-3.3799532851801817) -- (5.596810089735671,5.37996412094186);
\draw [fill=white] (5.5873202253576,0.93376402396454) circle (5pt);
\draw[color=black] (5.9,1.6) node {$o$};
\draw[color=black] (6.7,1.3) node {$\theta$};
\draw [fill=red] (1.1220360622690104,4.782027102699072) circle (5pt);
\draw[color=black] (0.6,5.12) node {$b_1$};
\draw [fill=blue] (10.068990913517998,4.762930933772083) circle (5pt);
\draw[color=black] (10.6,5.1) node {$a_1$};
\draw [fill=red] (10.05260438844619,-2.914499054769991) circle (5pt);
\draw[color=black] (10.6,-2.58) node {$b_2$};
\draw [fill=blue] (1.1056495371972013,-2.895402885843004) circle (5pt);
\draw[color=black] (0.6,-2.56) node {$a_2$};
\draw[color=black] (11.5,1.3) node {$x$};
\draw[color=black] (5.9,5.66) node {$y$};
\end{tikzpicture}
\caption{Proof of Theorem~\ref{thm4}}\label{fig2}
\end{figure}
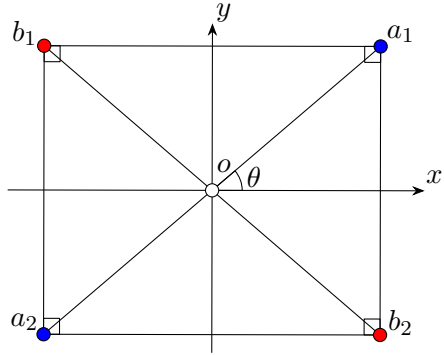

Let $N$ be any given $(A,B)$-network.
We have to show that the total length $\norm{N}_2\geq
4$.
Create a new $(A,B)$-network $N'$ by replacing each edge $\arc{v}{w}$ of $N$ by a broken edge, that is, an edge $\arc{v}{s}$ parallel to the $x$-axis, joined to an edge $\arc{s}{w}$ parallel to the $y$-axis, where $s$ is a new Steiner point of degree $(1,1)$.
Then
\begin{equation}\label{ineq}
\norm{N}_2\geq\norm{N}_{1(\theta)}=\norm{N'}_{1(\theta)},
\end{equation}
by Lemmas~\ref{lemma1} and \ref{lemma3}.
Each edge in $N'$ is in one of $4$ directions: the positive and negative $x$- and $y$-axes.
Consider any $(a_1,b_1)$-path $P$ in the new network $N'$.
The total $\norm{\cdot}_{1(\theta)}$-length of all the edges of $P$ in the direction of the negative $x$-axis is at least $\norm{a_1-b_1}_{1(\theta)} = (2\cos\theta)\cos\theta + 0\sin\theta= 2\cos^2\theta$.
Similarly, the total $\norm{\cdot}_{1(\theta)}$-length of all the edges of an $(a_2,b_2)$-path in the direction of the positive $x$-axis is also at least $2\cos^2\theta$.
In the same way, the total $\norm{\cdot}_{1(\theta)}$-length of all the edges of an $(a_2,b_1)$-path in the direction of the positive $y$-axis is at least $2\sin^2\theta$, and of all the edges of an $(a_1,b_2)$-path in the direction of the negative $y$-axis is also at least $2\sin^2\theta$.
Since we did not count any edge twice, we arrive at a lower bound for the $\norm{\cdot}_{1(\theta)}$-length of $N'$ of
\[ \norm{N'}_{1(\theta)} \geq 2\cos^2\theta+2\cos^2\theta+2\sin^2\theta+2\sin^2\theta = 4.\]
This, together with \eqref{ineq}, finishes the proof that $\norm{N}_2\geq 4$.
\end{proof}

As corollaries of the above Euclidean result, we obtain two results in the $\ell_1$-plane.
\begin{corollary}\label{cor1}
Consider the $\ell_1$-plane with unit ball with vertices $\pm e_1, \pm e_2$.
Let $A=\set{\pm e_1}$ be the set of sources and $B=\set{\pm e_2}$ the set of sinks.
Then the directed $(A,B)$-network with edge set $\set{\arc{e_1}{o},\arc{-e_1}{o},\arc{o}{e_2},\arc{o}{-e_2}}$ is shortest.
\end{corollary}
\begin{proof}
Let $N_0$ denote the network described in the corollary, and let $N$ be any $(A,B)$-network.
Note that $\norm{N}_2\leq\norm{N}_1$ for any network $N$.
By Theorem~\ref{thm4}, $\norm{N}_2\geq\norm{N_0}_2$.
It follows that $\norm{N}_1\geq\norm{N_0}_2=\norm{N_0}_1$.
\end{proof}
Since $\norm{x}_2\leq\norm{x}_1$, and because we can replace edges of shortest directed networks in the $\ell_1$-plane by broken edges without changing the length, it follows that in any shortest directed network in the $\ell_1$-plane, each vertex, including Steiner points, has indegree and outdegree at most $4$ each.
This can be attained.
\begin{corollary}\label{cor2}
In the $\ell_1$-plane, let $A=B=\set{\pm e_1,\pm e_2}$.
Then the $(A,B)$-network with the $8$ edges $\set{\arc{o}{\pm e_1}, \arc{\pm e_1}{o}, \arc{o}{\pm e_2}, \arc{\pm e_2}{o}}$ is shortest.
\end{corollary}
\begin{proof}
The proof is the similar to that of Theorem~\ref{thm4}: Consider all edges in the direction of the positive $x$-axis from $-e_1$ to $e_1$, etc.
\end{proof}

\section{Steiner points of degree 6}\label{sec:deg6}
In this section we prove part~\ref{33} of Theorem~\ref{thm:main}.
Proposition~\ref{lem:33} shows necessity of the angle condition on a Steiner point of degree $(3,3)$.
Then Theorem~\ref{thm6}, originally shown in \cite{Alfaro-Thesis}, shows sufficiency.
\begin{proposition}\label{lem:33}
Let $A=\set{a_1,a_2,a_3}$ and $B=\set{b_1,b_2,b_3}$ be two sets of three points each in the Euclidean plane such that $o\notin A\cup B$, and suppose that the network with edge set $\set{\arc{a_1}{o},\arc{a_2}{o},\arc{a_3}{o},\arc{o}{b_1},\arc{o}{b_2},\arc{o}{b_3}}$ is a shortest $(A,B)$-network.
Then the incoming and outgoing edges alternate, with consecutive directed edges at $60^\circ$ angles.
\end{proposition}
\begin{proof}
By Lemma~\ref{lem:local} we may assume that each edge has unit length.
By Lemma~\ref{lem:FT3}.\ref{lem:FT3:1} we know that the incoming edges are at $120^\circ$ angles, and the outgoing edges are at $120^\circ$ angles.
Denote the smallest angle between an incoming and outgoing edge by $\theta$.
Then $\theta\leq 60^\circ$.
We show that $\theta=60^\circ$ by contradiction.
If $\theta< 60^\circ$, then the network can be shortened, as shown in Fig.~\ref{fig:shorten33}.
In particular, we then have that the intersections $s_1$ of $a_2b_2$ and $a_3b_3$, $s_2$ of $a_1b_1$ and $a_3b_3$, and $s_3$ of $a_1b_1$ and $a_2b_2$ are distinct and distinct from $o$, and it follows from the triangle inequality that the network with edge set
\[\set{\arc{a_1}{s_2},\arc{a_2}{s_3},\arc{a_3}{s_1},\arc{s_3}{b_1},\arc{s_1}{b_2},\arc{s_2}{b_3},\arc{s_1}{s_2},\arc{s_2}{s_3},\arc{s_3}{s_1}}\]
is strictly shorter than the original network.
\begin{figure}
\definecolor{qqwuqq}{rgb}{0.,0.39215686274509803,0.}
\definecolor{qqqqff}{rgb}{0.,0.,1.}
\definecolor{ffqqqq}{rgb}{1.,0.,0.}
\centering
\begin{tikzpicture}[line cap=round,line join=round,>=Stealth,scale=1.1]
\draw [shift={(4.402975002933996,3.3969728569752)},color=qqwuqq,fill=qqwuqq,fill opacity=0.1] (0,0) -- (90.:0.3454071453719966) arc (90.:120.:0.3454071453719966) -- cycle;
\draw [->,line width=1pt] (4.402975002933996,3.3969728569752)-- (2.168629461170144,2.106972856975201);
\draw [->,line width=1pt] (4.402975002933996,3.3969728569752)-- (6.637320544697848,2.106972856975201);
\draw [->,line width=1pt] (4.402975002933996,3.3969728569752)-- (4.402975002933996,5.9769728569752);
\draw [->,line width=1pt] (3.1129750029339958,5.631318398739052)-- (4.402975002933996,3.3969728569752);
\draw [->,line width=1pt] (3.1129750029339953,1.1626273152113487)-- (4.402975002933996,3.3969728569752);
\draw [->,line width=1pt] (6.982975002933996,3.396972856975199)-- (4.402975002933996,3.3969728569752);
\draw [->,gray,line width=1pt] (9.843523677209324,3.4117549960140092)-- (7.609178135445473,2.12175499601401);
\draw [->,gray,line width=1pt] (9.843523677209324,3.4117549960140092)-- (12.077869218973177,2.12175499601401);
\draw [->,gray,line width=1pt] (9.843523677209324,3.4117549960140092)-- (9.843523677209324,5.991754996014009);
\draw [->,gray,line width=1pt] (8.553523677209325,5.646100537777861)-- (9.843523677209324,3.4117549960140092);
\draw [->,gray,line width=1pt] (8.553523677209323,1.1774094542501583)-- (9.843523677209324,3.4117549960140092);
\draw [->,gray,line width=1pt] (12.423523677209324,3.4117549960140074)-- (9.843523677209324,3.4117549960140092);
\draw [->,line width=1pt] (8.553523677209325,5.646100537777861)-- (9.497869218973177,4.701754996014009);
\draw [->,line width=1pt] (9.497869218973177,4.701754996014009)-- (9.843523677209324,5.991754996014009);
\draw [->,line width=1pt] (8.899178135445474,2.467409454250158)-- (7.609178135445473,2.12175499601401);
\draw [->,line width=1pt] (8.553523677209323,1.1774094542501583)-- (8.899178135445474,2.467409454250158);
\draw [->,line width=1pt] (12.423523677209324,3.4117549960140074)-- (11.133523677209327,3.0661005377778605);
\draw [->,line width=1pt] (11.133523677209327,3.0661005377778605)-- (12.077869218973177,2.12175499601401);
\draw [->,line width=1pt] (9.497869218973177,4.701754996014009)-- (11.133523677209327,3.0661005377778605);
\draw [->,line width=1pt] (11.133523677209327,3.0661005377778605)-- (8.899178135445474,2.467409454250158);
\draw [->,line width=1pt] (8.899178135445474,2.467409454250158)-- (9.497869218973177,4.701754996014009);
\draw [fill=qqqqff] (4.402975002933996,5.9769728569752) circle (2.5pt);
\draw (4.547280781505375,6.25) node {$b_2$};
\draw [fill=white] (4.402975002933996,3.3969728569752) circle (2.5pt);
\draw (4.57,3.6) node {$o$};
\draw [fill=qqqqff] (2.168629461170144,2.106972856975201) circle (2.5pt);
\draw (2.0948900493641984,2.3880071126328284) node {$b_1$};
\draw [fill=qqqqff] (6.637320544697848,2.106972856975201) circle (2.5pt);
\draw (6.8,2.4) node {$b_3$};
\draw [fill=ffqqqq] (3.1129750029339958,5.631318398739052) circle (2.5pt);
\draw (3.2117064860669875,5.83056499484041) node {$a_3$};
\draw (4.27,3.95) node {$\theta$};
\draw [fill=ffqqqq] (3.1129750029339953,1.1626273152113487) circle (2.5pt);
\draw (2.9,1.3632992480292343) node {$a_2$};
\draw [fill=ffqqqq] (6.982975002933996,3.396972856975199) circle (2.5pt);
\draw (7.11,3.65) node {$a_1$};

\draw [fill=qqqqff] (9.843523677209324,5.991754996014009) circle (2.5pt);
\draw (9.993200106870521,6.25) node {$b_2$};
\draw [gray,fill=white] (9.843523677209324,3.4117549960140092) circle (2.5pt);
\draw [draw=gray] (10,3.6) node {$o$};
\draw [fill=qqqqff] (7.609178135445473,2.12175499601401) circle (2.5pt);
\draw (7.5408093747293465,2.3995206841452283) node {$b_1$};
\draw [fill=qqqqff] (12.077869218973177,2.12175499601401) circle (2.5pt);
\draw (12.3,2.4) node {$b_3$};
\draw [fill=ffqqqq] (8.553523677209325,5.646100537777861) circle (2.5pt);
\draw (8.657625811432135,5.84207856635281) node {$a_3$};
\draw [fill=ffqqqq] (8.553523677209323,1.1774094542501583) circle (2.5pt);
\draw (8.3,1.374812819541634) node {$a_2$};
\draw [fill=ffqqqq] (12.423523677209324,3.4117549960140074) circle (2.5pt);
\draw (12.5,3.65) node {$a_1$};
\draw [fill=white] (9.497869218973177,4.701754996014009) circle (2.5pt);
\draw (9.42,5.06) node {$s_1$};
\draw [fill=white] (8.899178135445474,2.467409454250158) circle (2.5pt);
\draw (8.7,2.6) node {$s_3$};
\draw [fill=white] (11.133523677209327,3.0661005377778605) circle (2.5pt);
\draw (11.55,2.95) node {$s_2$};
\draw[color=black] (7.4,4.4) node {\LARGE$\leadsto$};
\end{tikzpicture}
\caption{Shortening a network with a Steiner point of degree $(3,3)$ if some angle is $<60^\circ$}\label{fig:shorten33}
\end{figure}
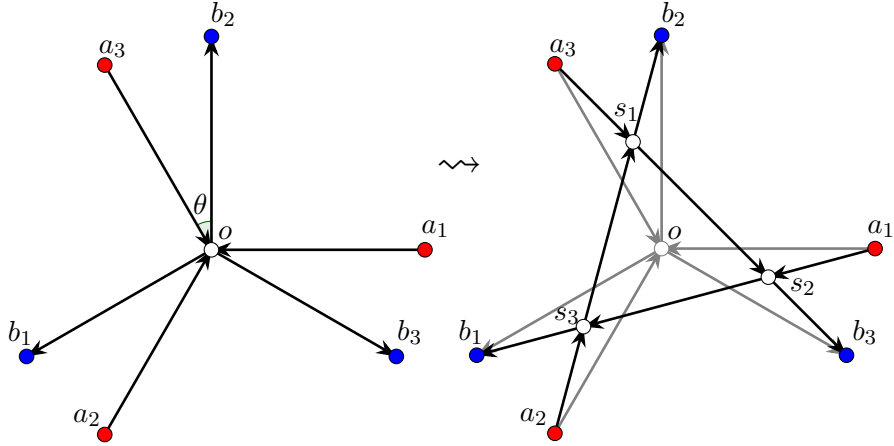
\end{proof}

\begin{theorem}[{\cite[Theorem~3.2]{Alfaro-Thesis}}]\label{thm6}
Let $A=\set{a_1,a_2,a_3}$ and $B=\set{b_1,b_2,b_3}$ be sets of points in the Euclidean plane such that $\myangle a_1ob_2 = \myangle b_2oa_3 = \myangle a_3ob_1 = \myangle b_1oa_2 = \myangle a_2ob_3 = \myangle b_3oa_1 = 60^\circ$.
Then the network with edge set $\set{\arc{a_1}{o},\arc{a_2}{o},\arc{a_3}{o},\arc{o}{b_1},\arc{o}{b_2},\arc{o}{b_3}}$ is shortest among all directed $(A,B)$-networks.
\end{theorem}
\begin{proof}
By Lemma~\ref{lem:local}, we may assume that the lengths $\norm{a_i-o}_2$ and $\norm{b_i-o}_2$ are all the same.
Thus $a_1b_2a_3b_1a_2b_3$ is a regular hexagon.
We may also assume without loss of generality that $o$ is the origin and $\norm{a_i}_2=\norm{b_i}_2=1$, $i=1,2,3$, as in Fig.~\ref{fig3}.
Denote the $(A,B)$-network described in the statement of the theorem by $N_0$.
Let $N$ be any given $(A,B)$-network.
We have to show that $\norm{N}_2\geq\norm{N_0}_2$.

By Lemma~\ref{lemma2}, $\norm{N}_2\geq\norm{N}_H$, and by Lemma~\ref{lemma3}, we can replace each edge of $N$ by a broken edge consisting of two edges parallel to one of the main diagonals $a_1'b_2'$, $a_2'b_3'$, $a_3'b_1'$ of $B_H$ (Fig.~\ref{fig3}) to create a new network $N'$ with all edges in one of the $6$ directions $a_i', b_i'$, and of the same length $\norm{N}_H=\norm{N'}_H$.

Consider any $(a_1,b_1)$-path $P$ in $N'$.
The only vectors on this path with negative $x$-component are those in the directions of $a_1'$ and $b_1'$.
If we project this path orthogonally onto the $x$-axis, the $\norm{\cdot}_H$-distance does not increase, by Lemma~\ref{lemma4}.
It follows that the total $\norm{\cdot}_H$-length of the edges in the directions of $a_1'$ and $b_1'$ is at least $2$.
By symmetry, the total $\norm{\cdot}_H$-length of the edges on an $(a_2,b_2)$-path in the directions of $a_2'$ and $b_2'$ is at least $2$, and the total $\norm{\cdot}_H$-length of the edges on an $(a_3,b_3)$-path in the directions of $a_3'$ and $b_3'$ is at least $2$.
Since we did not count any edge more than once, we obtain that $\norm{N'}_H\geq 6$.
If we put all the inequalities together, we obtain $\norm{N}_2\geq 6 = \norm{N_0}_2$.
\end{proof}

\begin{corollary}\label{cor3}
Consider the normed plane $(\bR^2,\norm{\cdot}_h)$ with unit ball the regular hexagon $B_h=a_1b_2a_3b_1a_2b_3$.
Let $A=\set{a_1,a_2,a_3}$ and $B=\set{b_1,b_2,b_3}$.
Then the network with edge set $\set{\arc{a_1}{o},\arc{a_2}{o},\arc{a_3}{o},\arc{o}{b_1},\arc{o}{b_2},\arc{o}{b_3}}$ is shortest among all directed $(A,B)$-networks in $(\bR^2,\norm{\cdot}_h)$.
\end{corollary}
\begin{proof}
The proof is similar to that of Corollary~\ref{cor1}.
Since $B_h\subseteq B_2$, we have that $\norm{x}_2\leq\norm{x}_h$ for all $x\in\bR^2$.
Let $N_0$ denote the network described in the corollary, and let $N$ be any $(A,B)$-network.
By Lemma~\ref{lemma2}, $\norm{N}_2\leq\norm{N}_h$.
By Theorem~\ref{thm6}, $\norm{N}_2\geq\norm{N_0}_2$.
It follows that $\norm{N}_h\geq\norm{N_0}_2=\norm{N_0}_h$.
\end{proof}

\section{Steiner points of degree 5}\label{sec:deg5}
\begin{theorem}
Let $a_1, a_2, a_3, b_1, b_2, s$ be points in the Euclidean plane.
Let $A=\{a_1,a_2,a_3\}$ and $B=\{b_1,b_2\}$.
Then the network with edges $\arc{a_i}{s}$, $i=1,2,3$ and $\arc{s}{b_i}$, $i=1,2$, is a shortest $(A,B)$-network if and only if $\myangle a_isa_j=120^\circ$ for all $1\leq i < j\leq 3$ and $\myangle b_1sb_2=180^\circ$.
\end{theorem}
\begin{proof}
We first show that the condition is necessary for a Steiner point $s$ of degree $(3,2)$.
By Lemmas~\ref{lem:FT2} and \ref{lem:FT3}, the three incoming directed edges $\arc{a_i}{s}$ are pairwise at $120^\circ$ angles, and the two outgoing directed edges $\arc{s}{b_i}$ have to be at an angle of $\geq 120^\circ$.
Therefore, $\arc{s}{b_1}$ and $\arc{s}{b_2}$ lie in different (closed) angles spanned by pairs of incoming edges.
If $\myangle b_1sb_2\neq 180^\circ$, then, assuming without loss of generality that the segment $b_1b_2$ intersects the segment $a_1s$ at a point $s'$, the network can be shortened by replacing $\arc{s}{b_1}$, $\arc{s}{b_2}$ and $\arc{a_1}{s}$ by $\arc{s}{s'}$, $\arc{a_1}{s'}$, $\arc{s'}{b_1}$, $\arc{s'}{b_2}$ as in Fig.~\ref{fig5}.
\begin{figure}
\definecolor{qqqqff}{rgb}{0.,0.,1.}
\definecolor{ffffff}{rgb}{1.,1.,1.}
\definecolor{ffqqqq}{rgb}{1.,0.,0.}
\centering
\begin{tikzpicture}[line cap=round,line join=round,>=Stealth,x=0.35cm,y=0.35cm]
\draw [->,line width=1pt] (2.,4.) -- (1.9982402614052646,-1.2350546258915038);
\draw [->,line width=1pt] (-2.536329904113388,-3.8510579605101936) -- (1.9982402614052646,-1.2350546258915043);
\draw [->,line width=1pt] (1.9982402614052646,-1.2350546258915043) -- (-2.6429330267308284,0.14657871045583426);
\draw [->,line width=1pt] (6.53105068832918,-3.854105917164317) -- (1.9982402614052646,-1.2350546258915043);
\draw [->,line width=1pt] (1.9982402614052646,-1.2350546258915043) -- (6.703493991243839,0.49146531628515744);
\draw [line width=1pt,dash pattern=on 4pt off 4pt] (-2.6429330267308284,0.14657871045583418)-- (6.703493991243839,0.49146531628515755);
\draw [line width=1pt] (1.9982402614052646,-1.2350546258915043)-- (2.,4.);
\draw [->,line width=1pt] (11.4833106228486,-3.9027909513845724) -- (16.017880788367254,-1.2867876167658956);
\draw [->,line width=1pt] (20.55069121529116,-3.9058389080387155) -- (16.017880788367254,-1.2867876167658956);
\draw [line width=1pt] (11.37670750023117,0.09484571958143506)-- (20.723134518205818,0.43973232541075846);
\draw [line width=1pt] (16.017880788367254,-1.2867876167658956)-- (16.019640526962,3.948267009125604);
\draw [line width=1pt,dash pattern=on 4pt off 4pt] (11.37670750023117,0.09484571958143506)-- (16.017880788367254,-1.2867876167658956);
\draw [line width=1pt,dash pattern=on 4pt off 4pt] (16.017880788367254,-1.2867876167658956)-- (20.723134518205818,0.43973232541075846);
\draw [->,line width=1pt] (16.017880788367254,-1.2867876167658956) -- (16.018402792854523,0.26612598868192805);
\draw [->,line width=1pt] (16.018402792854523,0.2661259886819281) -- (11.37670750023117,0.09484571958143506);
\draw [->,line width=1pt] (16.018402792854523,0.2661259886819281) -- (20.723134518205818,0.43973232541075846);
\draw [->,line width=1pt] (16.019640526962,3.948267009125604) -- (16.018402792854523,0.26612598868192805);
\draw [fill=ffqqqq] (6.53105068832918,-3.854105917164317) circle (2.5pt);
\draw [fill=ffqqqq] (2.,4.) circle (2.5pt) node[right=2pt] {$a_1$};
\draw [fill=ffqqqq] (-2.536329904113388,-3.8510579605101936) circle (2.5pt);
\draw [fill=ffffff] (1.9982402614052646,-1.2350546258915043) circle (2.5pt) node[below=2pt] {$s$};
\draw [fill=qqqqff] (-2.6429330267308284,0.14657871045583418) circle (2.5pt) node[left=1pt] {$b_1$};
\draw [fill=qqqqff] (6.703493991243839,0.49146531628515755) circle (2.5pt) node[right=2pt] {$b_2$};
\draw [fill=ffqqqq] (20.55069121529116,-3.9058389080387155) circle (2.5pt);
\draw [fill=ffqqqq] (16.019640526962,3.948267009125604) circle (2.5pt) node[right=2pt] {$a_1$};
\draw [fill=ffqqqq] (11.4833106228486,-3.9027909513845724) circle (2.5pt);
\draw [fill=ffffff] (16.017880788367254,-1.2867876167658956) circle (2.5pt) node[below=2pt] {$s$};
\draw [fill=qqqqff] (11.37670750023117,0.09484571958143506) circle (2.5pt) node[left=1pt] {$b_1$};
\draw [fill=qqqqff] (20.723134518205818,0.43973232541075846) circle (2.5pt) node[right=2pt] {$b_2$};
\draw [fill=ffffff] (16.018402792854523,0.2661259886819281) circle (2.5pt) node[above right=2pt] {$s'$};
\draw[color=black] (8.9,2) node {\LARGE$\leadsto$};
\end{tikzpicture}
\caption{If $b_1,s,b_2$ are not collinear, the network can be shortened.}\label{fig5}
\end{figure}
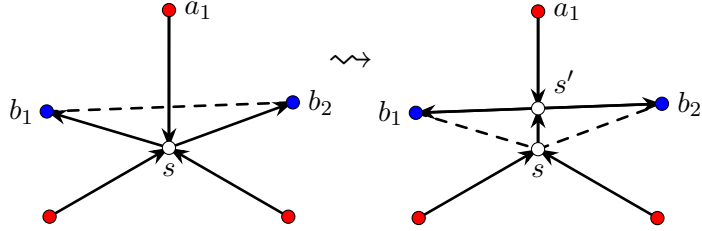
Therefore, $\myangle b_1sb_2 = 180^\circ$.

To show the converse, suppose that $a_1,a_2,a_3$ satisfy $\myangle a_isa_j=120^\circ$ for all distinct $i,j$, and $b_1,b_2$ satisfy $\myangle b_1sb_2=180^\circ$.
By Lemma~\ref{lem:local}, we may assume that the neighbours of $s$ lie on the unit circle centred at $s$.
By relabelling, we may also assume that $b_1$ lies inside $\myangle a_1 o a_2$ and $b_2$ inside $\myangle b_2 o a_3$.
By a limit argument we may assume that $A=\{a_1,a_2,a_3\}$ and $B=\{b_1,b_2\}$ are disjoint.
We want to show that the network with directed edges $\arc{a_i}{s}$, $i=1,2,3$ and $\arc{o}{b_j}$, $j=1,2$, is a shortest $(A,B)$-network.
Let $G$ be a simple shortest $(A,B)$-network.
We will show that $\norm{G}_2\geq 5$.
(We will only use the assumption that $G$ is shortest in Case~\ref{case2.3} in the last part of the proof).

If the underlying undirected graph of $G$ has a cycle, then we can reorient some edges of the cycle in $G$ so that the resulting digraph stays an $(A,B)$-network.
Among all reorientations of $G$ that are still $(A,B)$-networks, we choose one with an $(a_2,b_2)$-path $P$ that minimizes the area of the region $\myGamma$ bounded by $P$ and the segments $a_2a_3$ and $a_3b_2$ (Fig.~\ref{fig52}).
\begin{figure}
\centering
\definecolor{qqqqff}{rgb}{0.,0.,1.}
\definecolor{ffffff}{rgb}{1.,1.,1.}
\definecolor{ffqqqq}{rgb}{1.,0.,0.}
\begin{tikzpicture}[scale=0.8,line cap=round,line join=round,>=Stealth,x=1cm,y=1cm]
\fill [blue!15!white] (-1.62,-2.08) .. controls (0,0) and (1,-1) .. (2.12,-0.5) .. controls (3.24, 0) and (3,0.5) .. (4.32,0.64) .. controls (5.64,0.78) and (6.5, 0) .. (7.6,1.26) -- (7.,-2.14) -- cycle; 
\draw [blue!90!black] (3.6,-1.1) node {\LARGE $\myGamma$};
\draw[blue!90!black,dashed,line width=1.2pt] (7.6,1.26) -- (7.,-2.14) -- (-1.62,-2.08);

\draw [->,line width=1.2pt] (-1.62,-2.08) .. controls (0,0) and (1,-1) .. (2.12,-0.5) node [midway,below] {$P_1$}; 
\draw [->,line width=1.2pt] (2.12,-0.5) .. controls (3.24, 0) and (3,0.5) .. (4.32,0.64);
\draw [->,line width=1.2pt] (4.32,0.64) .. controls (5.64,0.78) and (6.5, 0) .. (7.6,1.26) node [midway,above] {$P_3$}; 
\draw [->,line width=1.2pt] (7.,-2.14) .. controls (5,-1) and (3,-1.5) .. (2.12,-0.5) node [near start,above] {$Q_1$}; 
\draw [->,line width=1.2pt] (4.32,0.64) .. controls (0.92,4.2) and (-0.6,-3.8).. (-2.18,-0.52) node [above,near start] {$Q_3$}; 

\draw [->,line width=1.2pt] (2.7419615242270683,5.355138980621861) .. controls (5.1,4) and (2,3) .. (5,1.6) node [midway,right] {$R$}; 
\draw [fill=ffqqqq] (-1.62,-2.08) circle (3pt); 
\draw[color=black] (-2.11,-1.98) node {$a_2$};
\draw [fill=ffffff] (2.12,-0.5) circle (3pt); 
\draw[color=black] (2,-0.15) node {$p$};
\draw [fill=ffqqqq] (7.,-2.14) circle (3pt); 
\draw[color=black] (7.5,-2) node {$a_3$};
\draw [fill=ffffff] (4.32,0.64) circle (3pt); 
\draw[color=black] (4.5,1) node {$q$};
\draw[color=black] (2.8,0.45) node {$P_2$};
\draw [fill=qqqqff] (-2.18,-0.52) circle (3pt);
\draw[color=black] (-2.5,-0.1) node {$b_1$};
\draw [fill=qqqqff] (7.6,1.26) circle (3pt); 
\draw[color=black] (7.4,1.7) node {$b_2$};
\draw [fill=ffqqqq] (2.7419615242270683,5.355138980621861) circle (3pt);
\draw[color=black] (3.2,5.46) node {$a_1$};
\end{tikzpicture}
\caption{Analysing the network $G$}\label{fig52}
\end{figure}
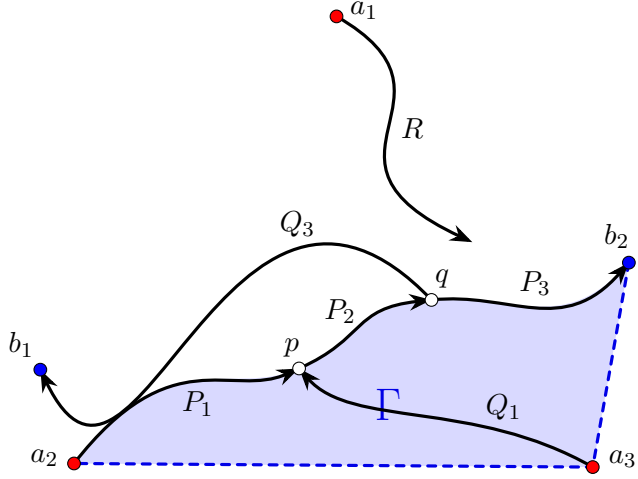
Let $Q$ be an $(a_3,b_1)$-path.
By Lemma~\ref{lem:convex}, $Q$ crosses $P$.
Let $p$ be the first vertex on $Q$ that is also on $P$.
Let $Q_1$ be the subpath of $Q$ from from $a_3$ to $p$.
Let $P_1$ be the part of $P$ from $a_2$ to $p$.
Let $q$ be the last vertex on $Q$ that is also on the part of $P$ from $p$ to $b_2$.
Let $P_2$ be the part of $P$ from $p$ to $q$, and $P_3$ the part of $P$ from $q$ to $b_2$.
We choose $Q$ such that the part of $Q$ from $p$ to $q$ coincides with $P_2$.
Let $Q_3$ be the part of $Q$ from $q$ to $b_1$.
Then no edge of $Q_3$ can be in the interior of the region $\myGamma$, otherwise $Q_3$ will cross $P$ again, thus making a cycle which can be reoriented to make $\myGamma$ smaller, which would contradict the minimality of $\myGamma$.
It is possible for vertices of $Q_3$ other than $q$ to lie on $P_1$.
Let $R$ be an $(a_1,b_2)$-path without repeated vertices.
Let $r$ be the first point on $R$ that is on $P_3\cup Q_3$.
We distinguish between two cases, depending on whether $r$ is on $P_3$ (Fig.~\ref{fig53}) or $Q_3$ (Fig.~\ref{fig54}).
\begin{case}
\item\label{case1} 
$r$ is on $P_3$ (including the case $r=q$).
See Fig.~\ref{fig53}.
\begin{figure}
\definecolor{qqqqff}{rgb}{0.,0.,1.}
\definecolor{ffffff}{rgb}{1.,1.,1.}
\definecolor{ffqqqq}{rgb}{1.,0.,0.}
\centering
\begin{tikzpicture}[line cap=round,line join=round,>=Stealth,x=0.7cm,y=0.7cm]
\fill [blue!15!white] (2.7419615242270683,5.355138980621861) -- (7.6,1.26) -- (5.54,1.64) -- (3.26,3.54) -- cycle;
\draw [blue!90!black,dashed,line width=1.2pt] (2.7419615242270683,5.355138980621861) -- (7.6,1.26);
\draw [blue!90!black] (4.8,2.9) node {$\myDelta$};
\draw [->,line width=1pt] (-1.62,-2.08) -- (2.12,-0.5);
\draw [->,line width=1pt] (7.,-2.14) -- (2.12,-0.5);
\draw [->,line width=1pt] (2.12,-0.5) -- (3.8,0.9);
\draw [->,line width=1pt] (3.8,0.9) -- (1.22,1.58);
\draw [->,line width=1pt] (1.22,1.58) -- (-2.18,-0.52);
\draw [->,line width=1pt] (3.8,0.9) -- (5.54,1.64);
\draw [->,line width=1pt] (5.54,1.64) -- (7.6,1.26);
\draw [->,line width=1pt] (2.7419615242270683,5.355138980621861) -- (3.26,3.54);
\draw [->,line width=1pt] (3.26,3.54) -- (1.22,1.58);
\draw [->,line width=1pt] (3.26,3.54) -- (5.54,1.64);
\draw [fill=ffqqqq] (-1.62,-2.08) circle (2.5pt);
\draw[color=black] (-2.2,-2.1) node {$a_2$};
\draw [fill=ffffff] (2.12,-0.5) circle (2.5pt);
\draw[color=black] (2.1,-1.0) node {$p$};
\draw [fill=ffqqqq] (7.,-2.14) circle (2.5pt);
\draw[color=black] (7.23,-1.72) node {$a_3$};
\draw [fill=ffffff] (3.8,0.9) circle (2.5pt);
\draw[color=black] (4.1,0.5) node {$q$};
\draw [fill=ffffff] (1.22,1.58) circle (2.5pt);
\draw[color=black] (1.26,1.1) node {$t$};
\draw [fill=qqqqff] (-2.18,-0.52) circle (2.5pt);
\draw[color=black] (-2.29,0) node {$b_1$};
\draw [fill=ffffff] (5.54,1.64) circle (2.5pt);
\draw[color=black] (5.8,2.01) node {$r$};
\draw [fill=qqqqff] (7.6,1.26) circle (2.5pt);
\draw[color=black] (7.6,0.8) node {$b_2$};
\draw [fill=ffffff] (3.26,3.54) circle (2.5pt);
\draw[color=black] (3.6,3.81) node {$s$};
\draw [fill=ffqqqq] (2.7419615242270683,5.355138980621861) circle (2.5pt);
\draw[color=black] (3.3,5.46) node {$a_1$};
\end{tikzpicture}
\caption{Case~\ref{case1}}\label{fig53}
\end{figure}
Then without loss of generality, the part of $R$ from $r$ to $b_2$ coincides with $P$.
We may also assume that among all $(a_1,b_2)$-paths hitting $P_3\cup Q_3$ first in $P_3$, we have chosen one with $r$ closest to $q$ on the path $P_3$.
Let $S$ be an $(a_1,b_1)$-path without repeated vertices, and let $s$ be the last point on $S$ that is also on $R$.
Then $s$ has to be on the part of $R$ from $a_1$ to $r$, otherwise $s$ would be on the part of $P_3$ from $r$ to $b_2$, not including $r$, and then $S$ would either have to cross itself or enter the interior of $\myGamma$, which would contradict either the choice of $S$ or the minimality of $\myGamma$.
Again, because $S$ does not cross itself, the part of $S$ from $s$ to $b_1$ cannot enter the polygon $\myDelta$ bounded by the part of $R$ from $a_1$ to $r$, the part of $P_3$ from $r$ to $b_2$ and the segment $a_1b_2$.
Let $t$ be the first point of $S$ that is on $Q_3$.
We have now described all $(a_i,b_j)$-paths, so by minimality of $G$ we have that $G$ consists of the directed edges of non-zero length among $\arc{a_1}{s}, \arc{s}{t}, \arc{q}{t}, \arc{q}{r}, \arc{s}{r}, \arc{t}{b_1}, \arc{r}{b_2}, \arc{p}{q}, \arc{a_2}{p}, \arc{a_3}{p}$. 
By applying the triangle inequality four times, we find that the perimeter of the quadrilateral $qrst$ is at least the sum of its diagonals $\norm{s-q}_2+\norm{t-r}_2$.
We now forget the directions of the directed edges, and replace the perimeter of $qrst$ by the diagonals to obtain a geometric graph that splits into a tree that connects $a_1,a_2,a_3$ and an edge-disjoint path that connects $b_1$ and $b_2$.
By Lemma~\ref{lem:FT} and the triangle inequality, the tree will have length bounded below by the length of the tree with edges $a_1o$, $a_2o$, $a_3o$, which equals $3$.
The path between $b_1$ and $b_2$ is bounded below by the distance $\norm{b_1-b_2}_2=2$.
That is,
\begin{align*}
\norm{G}_2 &\geq \norm{a_1-s}_2 + (\norm{s-t}_2 + \norm{q-t}_2 + \norm{q-r}_2 + \norm{s-r}_2)\\
&\mathrel{\phantom{\geq}} \mbox{} + \norm{t-b_1}_2 + \norm{r-b_2}_2 + \norm{p-q}_2 + \norm{a_2-p}_2 + \norm{a_3-p}_2\\
&\geq \norm{a_1-s}_2 + (\norm{s-q}_2 + \norm{t-r}_2) + \norm{t-b_1}_2 + \norm{r-b_2}_2\\
&\mathrel{\phantom{\geq}} \mbox{}  + \norm{p-q}_2 + \norm{a_2-p}_2 + \norm{a_3-p}_2\\
&= (\norm{a_1-s}_2+\norm{s-q}_2+\norm{p-q}_2+\norm{a_2-p}_2+\norm{a_3-p}_2)\\
&\mathrel{\phantom{\geq}} \mbox{} + (\norm{t-b_1}_2+\norm{t-r}_2+\norm{r-b_2}_2)\\
&\geq (\norm{a_1-p}_2 + \norm{a_2-p}_2 + \norm{a_3-p}_2) + \norm{b_1-b_2}_2\\
&\geq (\norm{a_1-o}_2 + \norm{a_2-o}_2 + \norm{a_3-o}_2) + \norm{b_1-b_2}_2 = 1+1+1+2,
\end{align*}
where we have used Lemma~\ref{lem:FT} in the last inequality and the triangle inequality in the others.

\item\label{case2}
$r$ is on $Q_3$ (and $r\neq q$).
See Fig.~\ref{fig54}.
\begin{figure}
\definecolor{qqqqff}{rgb}{0.,0.,1.}
\definecolor{ffffff}{rgb}{1.,1.,1.}
\definecolor{ffqqqq}{rgb}{1.,0.,0.}
\centering
\begin{tikzpicture}[line cap=round,line join=round,>=Stealth,x=0.7cm,y=0.7cm]
\fill [blue!15!white] (2.7419615242270683,5.355138980621861) .. controls (2.5,4) and (1.8,3) .. (1.92,2.2) .. controls (4,1.5) and (3.5,1) .. (4.32,0.64) .. controls (5.64,0.78) and (6.5, 0) .. (7.6,1.26) -- cycle;
\draw [blue!90!black,dashed,line width=1.2pt] (2.7419615242270683,5.355138980621861) -- (7.6,1.26);
\draw [blue!90!black] (4.1,2.3) node {\Large $\myGamma_{\!4}$};
\fill [blue!15!white] (-2.18,-0.52) -- (-1.62,-2.08) .. controls (0,0) and (1,-1) .. (2.12,-0.5) .. controls (3.24, 0) and (3,0.5) .. (4.32,0.64) .. controls (3.5,1) and (4,1.5) .. (1.92,2.2) .. controls (1.7,1.5) and (0.7,2) .. (-0.28,1.3) .. controls (-1.26,0.6) and (-2,-0.3) .. cycle;
\fill [blue!15!white] (1.92,2.2) .. controls (1.7,1.5) and (0.7,2) .. (-0.28,1.3) .. controls (-1.26,0.6) and (-2,-0.3) .. (-2.18,-0.52) -- (2.7419615242270683,5.355138980621861) .. controls (2.5,4) and (1.8,3) .. cycle;
\draw [blue!90!black,dashed,line width=1.2pt] (2.7419615242270683,5.355138980621861) -- (-2.18,-0.52);
\draw [blue!90!black,dashed,line width=1.2pt] (-2.18,-0.52) -- (-1.62,-2.08);
\draw [blue!90!black] (1,0.4) node {\Large $\myGamma_{\!2}$};
\draw [blue!90!black] (1.2,2.5) node {\Large $\myGamma_{\!0}$};
\draw [->,line width=1.2pt] (-1.62,-2.08) .. controls (0,0) and (1,-1) .. (2.12,-0.5) node [midway,below] {$P_1$}; 
\draw [->,line width=1.2pt] (2.12,-0.5) .. controls (3.24, 0) and (3,0.5) .. (4.32,0.64); 
\draw [->,line width=1.2pt] (4.32,0.64) .. controls (5.64,0.78) and (6.5, 0) .. (7.6,1.26) node [midway,above] {$P_3$}; 
\draw [->,line width=1.2pt] (7.,-2.14) .. controls (5,-1) and (3,-1.5) .. (2.12,-0.5) node [near start,above] {$Q_1$}; 
\draw [->,line width=1.2pt] (4.32,0.64) .. controls (3.5,1) and (4,1.5) .. (1.92,2.2); 
\draw [->,line width=1.2pt] (2.7419615242270683,5.355138980621861) .. controls (2.5,4) and (1.8,3) .. (1.92,2.2); 
\draw [->,line width=1.2pt] (1.92,2.2) .. controls (1.7,1.5) and (0.7,2) .. (-0.28,1.3); 
\draw [->,line width=1.2pt] (-0.28,1.3) .. controls (-1.26,0.6) and (-2,-0.3) .. (-2.18,-0.52); 
\draw [fill=ffqqqq] (-1.62,-2.08) circle (2.5pt);
\draw[color=black] (-2.11,-1.98) node {$a_2$};
\draw [fill=ffffff] (2.12,-0.5) circle (2.5pt);
\draw[color=black] (2.,-1.1) node {$p$};
\draw [fill=ffqqqq] (7.,-2.14) circle (2.5pt);
\draw[color=black] (7.23,-1.72) node {$a_3$};
\draw [fill=ffffff] (4.32,0.64) circle (2.5pt);
\draw[color=black] (4.4,0.2) node {$q$};
\draw [fill=ffffff] (1.92,2.2) circle (2.5pt);
\draw[color=black] (2.3,2.5) node {$r$};
\draw [fill=qqqqff] (-2.18,-0.52) circle (2.5pt);
\draw[color=black] (-2.37,0.1) node {$b_1$};
\draw [fill=qqqqff] (7.6,1.26) circle (2.5pt);
\draw[color=black] (7.9,0.8) node {$b_2$};
\draw [fill=ffqqqq] (2.7419615242270683,5.355138980621861) circle (2.5pt);
\draw[color=black] (3.3,5.46) node {$a_1$};
\draw [fill=ffffff] (-0.28,1.3) circle (2.5pt);
\draw[color=black] (-0.2,1.0) node {$s$};
\draw[color=black] (2.7,0.45) node {$P_2$};
\end{tikzpicture}
\caption{Case~\ref{case2}}\label{fig54}
\end{figure}
Let $s$ be the last point of $R$ on $Q_3$.
Without loss of generality, the part of $R$ from $r$ to $s$ coincides with $Q_3$.
By Lemma~\ref{lem:convex}, the region $\myGamma_{\!0}$ bounded by the part of $R$ from $a_1$ to $s$, the part of $Q_3$ from $s$ to $b_1$ and the segment $a_1b_1$, is connected.
The directed edge $e$ on $R$ following $s$ cannot be in $\myGamma_{\!0}$, because $R$ does not have repeated vertices and $s$ is the last point of $R$ on $Q_3$.
Therefore, $e$ is in the interior of either the region $\myGamma_{\!2}$ bounded by $Q_3$, the segment $b_1a_2$, and the part of $P$ from $a_2$ to $q$, or the region $\myGamma_{\!4}$ bounded by the part of $R$ from $a_1$ to $r$, the part of $Q_3$ from $q$ to $r$, $P_3$ and the segment $a_1b_2$ (and then $s=r$).
Let $t$ be the first point on $R$ that is also on $P$.
Without loss of generality, the part of $R$ from $t$ to $b_2$ coincides with $P$.
If $e$ is in the interior of $\myGamma_{\!2}$, then $t$ is either in $P_1$ or $P_2$, or the path from $s$ to $t$ hits $Q$ in the part from $q$ to $r$ before passing into $\myGamma_{\!4}$.

We thus have three subcases, depending on whether $t$ is on $P_1$, $P_2$, or $P_3$.

\begin{case}
\item\label{case2.2} 
$t\in P_2$.
See Fig.~\ref{fig522}.
By minimality of $G$, its edges are the ones of non-zero length among $\arc{a_1}{r}$, $\arc{r}{s}$, $\arc{s}{b_1}$, $\arc{s}{t}$, $\arc{a_2}{p}$, $\arc{p}{t}$, $\arc{t}{q}$, $\arc{q}{r}$, $\arc{q}{b_2}$, $\arc{a_3}{p}$.
\begin{figure}
\definecolor{qqqqff}{rgb}{0.,0.,1.}
\definecolor{ffffff}{rgb}{1.,1.,1.}
\definecolor{ffqqqq}{rgb}{1.,0.,0.}
\centering
\begin{tikzpicture}[line cap=round,line join=round,>=Stealth,x=0.7cm,y=0.7cm]
\draw [->,line width=1pt] (-1.62,-2.08) -- (2.12,-0.5);
\draw [->,line width=1pt] (7.,-2.14) -- (2.12,-0.5);
\draw [->,line width=1pt] (4.32,0.64) -- (1.92,2.2);
\draw [->,line width=1pt] (4.32,0.64) -- (7.6,1.26);
\draw [->,line width=1pt] (2.7419615242270683,5.355138980621861) -- (1.92,2.2);
\draw [->,line width=1pt] (1.92,2.2) -- (-0.28,1.3);
\draw [->,line width=1pt] (-0.28,1.3) -- (-2.18,-0.52);
\draw [->,line width=1pt] (2.12,-0.5) -- (3.02,0.12);
\draw [->,line width=1pt] (3.02,0.12) -- (4.32,0.64);
\draw [->,line width=1pt] (-0.28,1.3) -- (3.02,0.12);
\draw [fill=ffqqqq] (-1.62,-2.08) circle (2.5pt);
\draw[color=black] (-2.11,-1.98) node {$a_2$};
\draw [fill=ffffff] (2.12,-0.5) circle (2.5pt);
\draw[color=black] (2.1,-1.0) node {$p$};
\draw [fill=ffqqqq] (7.,-2.14) circle (2.5pt);
\draw[color=black] (7.23,-1.72) node {$a_3$};
\draw [fill=ffffff] (4.32,0.64) circle (2.5pt);
\draw[color=black] (4.5,0.3) node {$q$};
\draw [fill=ffffff] (1.92,2.2) circle (2.5pt);
\draw[color=black] (1.96,1.7) node {$r$};
\draw [fill=qqqqff] (-2.18,-0.52) circle (2.5pt);
\draw[color=black] (-2.37,0.18) node {$b_1$};
\draw [fill=qqqqff] (7.6,1.26) circle (2.5pt);
\draw[color=black] (7.6,0.8) node {$b_2$};
\draw [fill=ffqqqq] (2.7419615242270683,5.355138980621861) circle (2.5pt);
\draw[color=black] (3.3,5.4) node {$a_1$};
\draw [fill=ffffff] (-0.28,1.3) circle (2.5pt);
\draw[color=black] (-0.4,1.67) node {$s$};
\draw [fill=ffffff] (3.02,0.12) circle (2.5pt);
\draw[color=black] (3.32,-0.15) node {$t$};
\end{tikzpicture}
\caption{Case~\ref{case2.2}}\label{fig522}
\end{figure}
Then we finish as in Case~\ref{case1} by replacing the perimeter of the quadrilateral $rstq$ by its diagonals.

\item\label{case2.1} 
$t\in P_3$.
Then $R$ hits the part of $Q_3$ from $q$ to $r$ before hitting $P_3$ at $t$.
Let $u$ be the point on $Q_3$ where $R$ hits $Q_3$ first and let $v$ be the last point of $R$ on the part of $Q_3$ from $u$ to $r$ (Fig.~\ref{fig521}).
By minimality of $G$, its edges are the ones of non-zero length among $\arc{a_1}{r}$, $\arc{r}{s}$, $\arc{s}{b_1}$, $\arc{s}{u}$, $\arc{u}{v}$, $\arc{v}{r}$, $\arc{v}{t}$, $\arc{t}{b_2}$, $\arc{a_2}{p}$, $\arc{p}{q}$, $\arc{q}{u}$, $\arc{q}{t}$, $\arc{a_3}{p}$.
\begin{figure}
\definecolor{qqqqff}{rgb}{0.,0.,1.}
\definecolor{ffffff}{rgb}{1.,1.,1.}
\definecolor{ffqqqq}{rgb}{1.,0.,0.}
\centering
\begin{tikzpicture}[line cap=round,line join=round,>=Stealth,x=0.7cm,y=0.7cm]
\draw [->,line width=1pt] (-1.62,-2.08) -- (2.12,-0.5);
\draw [->,line width=1pt] (7.,-2.14) -- (2.12,-0.5);
\draw [->,line width=1pt] (2.12,-0.5) -- (4.36,0.24);
\draw [->,line width=1pt] (2.7419615242270683,5.355138980621861) -- (1.92,2.2);
\draw [->,line width=1pt] (1.92,2.2) -- (-0.28,1.3);
\draw [->,line width=1pt] (-0.28,1.3) -- (-2.18,-0.52);
\draw [->,line width=1pt] (-0.28,1.3) -- (3.88,1.2);
\draw [->,line width=1pt] (4.36,0.24) -- (3.88,1.2);
\draw [->,line width=1pt] (3.88,1.2) -- (3.26,1.98);
\draw [->,line width=1pt] (3.26,1.98) -- (1.92,2.2);
\draw [->,line width=1pt] (3.26,1.98) -- (5.84,1.1);
\draw [->,line width=1pt] (4.36,0.24) -- (5.84,1.1);
\draw [->,line width=1pt] (5.84,1.1) -- (7.6,1.26);
\draw [fill=ffqqqq] (-1.62,-2.08) circle (2.5pt);
\draw[color=black] (-2.11,-1.98) node {$a_2$};
\draw [fill=ffffff] (2.12,-0.5) circle (2.5pt);
\draw[color=black] (2.1,-1.0) node {$p$};
\draw [fill=ffqqqq] (7.,-2.14) circle (2.5pt);
\draw[color=black] (7.23,-1.72) node {$a_3$};
\draw [fill=ffffff] (4.36,0.24) circle (2.5pt);
\draw[color=black] (4.6,-0.1) node {$q$};
\draw [fill=ffffff] (1.92,2.2) circle (2.5pt);
\draw[color=black] (1.96,1.8) node {$r$};
\draw [fill=qqqqff] (-2.18,-0.52) circle (2.5pt);
\draw[color=black] (-2.37,0.1) node {$b_1$};
\draw [fill=qqqqff] (7.6,1.26) circle (2.5pt);
\draw[color=black] (7.45,0.8) node {$b_2$};
\draw [fill=ffqqqq] (2.7419615242270683,5.355138980621861) circle (2.5pt);
\draw[color=black] (3.3,5.4) node {$a_1$};
\draw [fill=ffffff] (-0.28,1.3) circle (2.5pt);
\draw[color=black] (-0.4,1.67) node {$s$};
\draw [fill=ffffff] (3.88,1.2) circle (2.5pt);
\draw[color=black] (3.6,0.8) node {$u$};
\draw [fill=ffffff] (3.26,1.98) circle (2.5pt);
\draw[color=black] (3.0,1.7) node {$v$};
\draw [fill=ffffff] (5.84,1.1) circle (2.5pt);
\draw[color=black] (5.9,0.7) node {$t$};
\end{tikzpicture}
\caption{Case~\ref{case2.1}}\label{fig521}
\end{figure}
As before, we finish as in Case~\ref{case1} by replacing the perimeter of the quadrilateral $rsuv$ by its diagonals.

\item\label{case2.3} 
$t\in P_1$ (Fig.~\ref{fig523}). 
\begin{figure}
\definecolor{qqqqff}{rgb}{0.,0.,1.}
\definecolor{ffffff}{rgb}{1.,1.,1.}
\definecolor{ffqqqq}{rgb}{1.,0.,0.}
\centering
\begin{tikzpicture}[line cap=round,line join=round,>=Stealth,x=0.7cm,y=0.7cm]
\draw [->,line width=1pt] (-1.62,-2.08) -- (2.12,-0.5);
\draw [->,line width=1pt] (7.,-2.14) -- (3.02,0.12); 
\draw [->,line width=1pt] (4.32,0.64) -- (1.92,2.2);
\draw [->,line width=1pt] (4.32,0.64) -- (7.6,1.26);
\draw [->,line width=1pt] (2.7419615242270683,5.355138980621861) -- (1.92,2.2);
\draw [->,line width=1pt] (1.92,2.2) -- (-0.28,1.3);
\draw [->,line width=1pt] (-0.28,1.3) -- (-2.18,-0.52);
\draw [->,line width=1pt] (2.12,-0.5) -- (3.02,0.12);
\draw [->,line width=1pt] (3.02,0.12) -- (4.32,0.64);
\draw [->,line width=1pt] (-0.28,1.3) -- (2.12,-0.5); 
\draw [fill=ffqqqq] (-1.62,-2.08) circle (2.5pt);
\draw[color=black] (-2.11,-1.98) node {$a_2$};
\draw [fill=ffffff] (2.12,-0.5) circle (2.5pt);
\draw[color=black] (2.1,-1.0) node {$t$};
\draw [fill=ffqqqq] (7.,-2.14) circle (2.5pt);
\draw[color=black] (7.23,-1.72) node {$a_3$};
\draw [fill=ffffff] (4.32,0.64) circle (2.5pt);
\draw[color=black] (4.5,0.2) node {$q$};
\draw [fill=ffffff] (1.92,2.2) circle (2.5pt);
\draw[color=black] (1.8,1.8) node {$r$};
\draw [fill=qqqqff] (-2.18,-0.52) circle (2.5pt);
\draw[color=black] (-2.37,0.1) node {$b_1$};
\draw [fill=qqqqff] (7.6,1.26) circle (2.5pt);
\draw[color=black] (7.7,0.7) node {$b_2$};
\draw [fill=ffqqqq] (2.7419615242270683,5.355138980621861) circle (2.5pt);
\draw[color=black] (3.3,5.4) node {$a_1$};
\draw [fill=ffffff] (-0.28,1.3) circle (2.5pt);
\draw[color=black] (-0.4,1.67) node {$s$};
\draw [fill=ffffff] (3.02,0.12) circle (2.5pt);
\draw[color=black] (3,-0.4) node {$p$};
\end{tikzpicture}
\caption{Case~\ref{case2.3}}\label{fig523}
\end{figure}
This case is slightly more complicated since we now have a pentagon $pqrst$ instead of a quadrilateral.
However, we show that in this case, at least one of the edges of the pentagon must be degenerate, and then we finish as before by replacing the edges of the resulting quadrilateral by its diagonals.
Suppose to the contrary that all five edges of $pqrst$ have non-zero lengths.
Since $\myangle b_1a_1b_2=90^\circ<120^\circ$, $\deg^+(a_1)=1$ and $\deg^-(a_1)=0$ by Lemmas~\ref{lem:FT3} and \ref{lem:convex}.
Thus $a_1\neq r$, hence $\myangle qrs=120^\circ$ by Lemma~\ref{lem:FT}.
Since $\myangle b_1a_2a_3=180^\circ-\myangle b_1a_1a_3=120^\circ-\myangle b_1a_1a_2<120^\circ$,
and similarly, $\myangle b_2a_3a_3<120^\circ$,
we obtain in the same way that $\myangle qpt=\myangle pts=120^\circ$.
Finally, we either have that $b_1\neq s$ and then also $\myangle tsr=120^\circ$, or $b_1=s$, and then $\myangle tsr=\myangle tb_1r\geq120^\circ$ by Lemma~\ref{lem:FT3}.
(In fact, equality has to hold since $\myangle a_1b_1a_2=120^\circ$.)
Similarly, $\myangle rqp\geq 120^\circ$.
It follows that the interior angle sum of $pqrst$ is at least $5\cdot120^\circ$, a contradiction.

Therefore, we have that at least one of the edges of $pqrst$ has zero length.
If either $p=q$, $p=t$ or $s=t$, then we replace the perimeter of the quadrilateral by its diagonals and finish as in Case~\ref{case1}.
If on the other hand, $r=q$ or $r=s$, then we can already split the underlying graph into two edge-disjoint connected subgraphs, one joining $a_1,a_2,a_3$, and the other a path joining $b_1$ and $b_2$.
\qedhere
\end{case}
\end{case}
\end{proof}

\subsection*{Acknowledgement}
We thank Frank Morgan for providing copies of references \cite{Alfaro-Thesis} and \cite{ACSS}, and the referee for valuable comments leading to an improved paper.

\end{document}